\documentclass{amsart}
\usepackage{graphicx}
\usepackage{amsfonts}
\usepackage{float} 
\newtheorem{tm}{Theorem}

\newtheorem{defi}{Definition}

\newtheorem{rem}{Remark}
\newtheorem{rems}{Remarks}
\newtheorem{lm}{Lemma}
\newtheorem{ex}{Example}
\newtheorem{problem}{Problem}

\newtheorem{prop}{Proposition}
\newtheorem{nota}{Notation}

\begin{document}

\title{Degree 5 polynomials and Descartes' rule of signs}
\author{Hassen Cheriha, Yousra Gati and Vladimir Petrov Kostov}
\address{Universit\'e C\^ote d'Azur, LJAD, France and 
University of Carthage, EPT - LIM, Tunisia}
\email{hassen.cheriha@gmail.com, hassan.cheriha@unice.fr}
\address{University of Carthage, EPT - LIM, Tunisia}
\email{yousra.gati@gmail.com} 
\address{Universit\'e C\^ote d'Azur, LJAD, France} 
\email{vladimir.kostov@unice.fr}

\begin{abstract}
For a univariate real polynomial without zero coefficients, Descartes' rule of signs (completed by an observation of Fourier) says that its numbers $pos$ of positive and $neg$ of negative roots (counted with multiplicity) are majorized respectively by the numbers $c$ and $p$ of sign changes and sign preservartions in the sequence of its coefficients, and that the differences $c-pos$ and $p-neg$ are even numbers. For degree 5 polynomials, it has been proved by A.~Albouy and Y.~Fu that there exist no such polynomials having three distinct positive and no negative roots and whose signs of the coefficients are $(+,+,-,+,-,-)$ (or having three distinct negative and no positive roots and whose signs of the coefficients are $(+,-,-,-,-,+)$). For degree 5 and when the leading coefficient is positive, these are all cases of numbers of positive and negative roots (all distinct) and signs of the coefficients which are compatible with Descartes' rule of signs, but for which there exist no such polynomials. We explain this non-existence and the existence in all other cases with $d=5$ by means of pictures showing the discriminant set of the family of polynomials $x^5+x^4+ax^3+bx^2+cx+d$ together with the coordinate axes.


\end{abstract}
\maketitle 

\section{Introduction}

Consider a univariate real polynomial $P(x):=\sum _{j=0}^da_jx^j$, $a_d\neq 0$, with $c$ sign changes in the sequence of its coefficients. The classical Descartes' rule of signs says that the number $pos$ of its positive roots is not larger than $c$, 
see~\cite{VJ}. Fourier (see~\cite{Fo}) has observed also that if roots are counted with multiplicity, then the number $c-pos$ is even. In the present paper we consider polynomials with all coefficients nonzero. In this case if one considers the polynomial $P(-x)$ and applies Descartes' rule to it, one finds that for the number $neg$ of 
negative roots of $P$, (counted with multiplicity) one obtains $neg\leq p$, where $p$ is the number of sign preservations in 
the sequence of coefficients (hence $c+p=d$); moreover, the number $p-neg$ is even. Descartes' rule of signs gives only necessary conditions about the possible values of the numbers $pos$ and $neg$ when the numbers $c$ and $p$ are known. To explain what sufficient conditions means we need the following definition:

\begin{defi}
{\rm For a given degree $d$, a {\em sign pattern (SP)} is a sequence of $d+1$ signs ($+$ or $-$). We assume the first of them to be a $+$, because without loss of generality we consider only monic polynomials. Given the degree $d$ and a SP, we denote by $c$ and $p$ the numbers of sign changes and sign preservations in the SP and we call the pair $(c,p)$ {\em Descartes' pair}. Any pair $(pos, neg)$ satisfying the conditions}

\begin{equation}\label{eqconditions}
pos\leq c~~~,~~~c-pos\in 2\mathbb{Z}~~~,~~~neg\leq p~~~,~~~p-neg\in 2\mathbb{Z}
\end{equation}

\noindent {\rm is called {\em admissible pair (AP)} for the given SP. In particular, the Descartes' pair is an AP.  A given couple (SP, AP) is {\em realizable} if there exists a monic degree $d$ polynomial the signs of whose coefficients define the given SP and which has exactly $pos$ positive and exactly $neg$ negative roots, all of them simple.}
\end{defi}
To give sufficient conditions in the context of Descartes' rule of signs means to give the answer to the following realization problem:

\begin{problem}\label{problem1}
For a given degree $d$, which couples (SP, AP) are realizable and which are not?
\end{problem}

The answer to this problem is known for $d\leq 8$. For $d\leq 3$, all couples (SP, AP) are realizable. For $d=4$, the answer to it is due to D.~Grabiner, see~\cite{Gr}, for $d=5$ and $6$ it is due to A.~Albouy and Y.~Fu, see~\cite{AlFu}, and for $d=7$ and $8$, it was given by J.~Forsg\aa rd, V.~P.~Kostov and B.~Z.~Shapiro, see~\cite{FoKoSh} and \cite{Czechoslovak}. 

\begin{rem} \label{remgen}
{\rm In order to reduce the number of couples (SP, AP) to be considered one can use the following  {\em $\mathbb{Z}_2\times 
\mathbb{Z}_2$-action}. Its first generator $g_1$ changes a given polynomial $P(x)$ to $(-1)^dP(-x)$ thereby changing every second sign of the SP and replacing the AP $(pos, neg)$ by the AP $(neg, pos)$. The second generator  $g_2$ changes $P(x)$ to $P^R(x):=x^dP(1/x)/P(0)$ which means reading the SP backward and preserving the AP (the roots of the {\em reverted polynomial} $P^R$ are the reciprocals of the roots of~$P$). The generators $g_1$ and $g_2$ are commuting involutions. Given a couple (SP, AP) (denoted by $\lambda$), the couples $\lambda$ and $g_1(\lambda )$ are always different, because the second signs of their SPs are different, but one might have $g_2(\lambda )=\lambda$ or $g_1g_2(\lambda )=\lambda$.

Thus orbits of the $\mathbb{Z}_2\times \mathbb{Z}_2$-action consist of $4$ or 
$2$ couples (SP, AP). E.g.
for $d=2$, one has the orbit $((+,-,-),(1,1))$, $((+,+,-),(1,1))$ of 
length $2$; for $d=3$, the orbit

$$\begin{array}{ccc}((+,+,+,-)~,~(1,2))~,&&((+,-,+,+)~,~(2,1))~,\\ \\
   ((+,-,-,-)~,~(1,2))~,&&((+,+,-,+)~,~(2,1))\end{array}$$
is of length~$4$. It is clear that all $4$ or $2$ couples (SP, AP) of a 
given orbit are simultaneously
(non)realizable.}
\end{rem}

\noindent In each of the cases $d=4$ and $d=5$ there is exactly one example of non-realizability of a couple (SP, AP) modulo the $\mathbb{Z}_2\times \mathbb{Z}_2$-action, namely

\begin{equation}\label{eqGrAlFu}
(~(+,+,-,+,+)~,~(2,0)~)~~~\, \, {\rm and}~~~\, \, 
(~(+,+,-,+,-,-)~,~(3,0)~)~,
\end{equation}
\noindent see~\cite{Gr} and~\cite{AlFu} respectively. For each of these two couples (SP, AP) one has $g_2(\lambda)=\lambda$, see Remark~\ref{remgen}, so they define orbits of length 2. For $d$ = 6, 7 and 8, there are respectively 4, 6 and 19 non-realizable cases modulo the $\mathbb{Z}_2\times \mathbb{Z}_2$-action, see~\cite{AlFu}, \cite{FoKoSh} and~\cite{Czechoslovak}.  

\begin{prop}\label{prophowmany}
For $d=5$, there are $22$ realizable and no non-realizable orbits of the $\mathbb{Z}_2\times \mathbb{Z}_2$-action of length $4$ and $13$ realizable and one non-realizable orbits of length~$2$.
\end{prop}

\noindent The proposition is proved in Section~\ref{secpictures} after Remarks~\ref{remscase}.

In \cite{KostovShapiroActaUMB} the {\em discriminant set} of the family of polynomials $x^4+x^3+ax^2+bx+c$ is represented (i.e. the set of values of the triple $(a,b,c)$ for which the polynomial has a multiple real root) and thus the non-realizability of the first of the two cases (\ref{eqGrAlFu}) is explained geometrically. In the present paper we give such an explanation for the non-realizability of the second of these cases and of the realizability of all other cases with $d=5$.
One can assume that the first two signs of the SP are $(+,+)$. Recall that $a_d=1$. The change $P(x)\mapsto P(a_{d-1}x)/a_{d-1}^d$  transforms $P(x)$ into $x^d+x^{d-1}+\cdots$, i.e. one can normalize the first two coefficients. So we consider the 4-parameter family of polynomials 
\begin{equation}\label{eqP}
P(x,a,b,c,d):=x^5+x^4+ax^3+bx^2+cx+d~,
\end{equation}
with $a$, $b$, $c$, $d\in \mathbb{R}$. We denote by $\Delta$ the {\em discriminant set}
\begin{equation}\label{setRD}
\Delta ~:=~\{ ~(a, b, c, d)\in \mathbb{R}^4~|~{\rm Res}(P, \partial P/\partial x, x)=0~\} ~.
\end{equation}
A polynomial of the family $P$ has a multiple real root exactly if $(a,b,c,d)\in \Delta$. 
Our aim is to explain by means of pictures of the set $\Delta$ why the second of the cases (\ref{eqGrAlFu}) is not realizable. These pictures are given in Section~\ref{secpictures}. In Section~\ref{secdiscrim} we remind some properties of the set $\Delta$ and we explain the notation used on the pictures. 

\section{Properties of the discriminant set \protect\label{secdiscrim}}

\noindent The set $\Delta$ partitions $\mathbb{R}^4\setminus \Delta$ into three open domains, in which a polynomial of the family $P$ has $5$, $3$ or $1$ simple real roots and hence $0$, $1$ or $2$ conjugate pairs respectively (for properties of discriminants see~\cite{Arnold}). On the figures these domains are indicated by the letters $h$, $t$ and $s$ respectively. We remind that polynomials of the domain $h$ (i.e. with all roots real) are called {\em hyperbolic}; the set of values of the parameters $(a,b,c,d)$ for which the polynomial $P$ is hyperbolic is called the {\em hyperbolicity domain} of the family (\ref{eqP}). The domain $h$, contrary to the domains $t$ and $s$, is not present on all figures, and when it is present, it is bounded; it is a curvilinear quadrigon or triangle, see part (2) of Remarks~\ref{remsfighyp}.  The set $\Delta$ and the coordinate hyperplanes together partition the set

\begin{equation}\label{eqRDelta}  
\mathbb{R}^4~\setminus ~\{ ~\Delta ~\cup ~\{ ~a=0~\} ~\cup ~\{ ~b=0~\} ~\cup ~\{ ~c=0~\} ~\cup ~\{ ~d=0~\} ~\}
\end{equation}
into open domains in each of which both the number of real roots and the signs of the coefficients of the polynomial $P$ remain the same; in fact, the real roots are distinct and nonzero hence the number of positive and negative roots is the same in each of the domains. The non-realizability of the second of the cases (\ref{eqGrAlFu}) is explained by the absence of the corresponding domain. 

\begin{rem}
{\rm It would be interesting to (dis)prove that each of the open domains of the set (\ref{eqRDelta}) is contractible
and that to each realizable case (SP, AP) there corresponds exactly one of these domains.
}
\end{rem}

\begin{rem}
{\rm For $d=6$, $7$ and $8$, the following {\em neighbouring property} holds true (the property can be checked directly
using the results of \cite{AlFu}, \cite{FoKoSh} and \cite{Czechoslovak}): For each two non-realizable orbits $C_0$, $C_*$ of the $\mathbb{Z}_2\times \mathbb{Z}_2$-action one can find a finite string of such orbits $C_1$, $C_2$, $\ldots$, $C_s=C_*$ such that for each two of the orbits $C_i$ and $C_{i+1}$ of this string there exist couples (SP, AP) $C_i^0\in C_i$, $C_{i+1}^0\in C_{i+1}$, such that either the SPs of $C_i^0$ and $C_{i+1}^0$ differ only by one sign and their APs are the same, or their SPs and one 
of the components of their APs are the same while the other components of the APs differ by $\pm 2$.

Example: for $d=6$, the non-realizable cases are the ones of the orbits of the following couples
(SP, AP), see~\cite{AlFu}:

$$\begin{array}{lcl}
(~(+,+,-,+,-,-,+)~,~(4,0)~) &,& (~(+,+,-,+,-,+,+) ~,~(2,0)~)  \\
(~(+,+,-,+,-,+,+)~,~(4,0)~) &{\rm and}& (~(+,+,-,+,+,+,+) ~,~(2,0)~)~. 
\end{array}$$

\noindent It is clear that they are neighbouring.

If the couples $C_i^0$ and $C_{i+1}^0$ were realizable, then they would correspond to two domains of the set (\ref{eqRDelta}) separated by a hypersurface, either by $\Delta$ or by one of the coordinate hyperplanes.

For $d=9$, the neighbouring property does not hold true. Indeed, for $d=9$, there exists a single non-realizable case (modulo the $\mathbb{Z}_2\times \mathbb{Z}_2$-action) with both components of the AP nonzero, this is the couple $C_{\sharp}:=((+,-,-,-,-,+,+,+,+,-),(1,6))$, see~\cite{YVH2s}. There exist non-realizable cases in which one of the components of the AP 
equals 0, see \cite{YVHd9}. However there are no non-realizable couples in which the AP equals $(1,8)$, $(1,4)$, $(3,6)$, $(8,1)$, $(4,1)$ or $(6,3)$. For $d=9$, there exist non-realizable cases in which one of the components of the AP equals 0,
see [5]. Hence if $C_0$ is an orbit of a couple (SP, AP)
with one of the components of the AP equal to $0$ and $C_*$ is the orbit of the couple $C_{\sharp}$, 
then one cannot construct the string of orbits $C_i$.}
\end{rem}
The set $\Delta$ is {\em stratified}. Its strata are defined by the {\em multiplicity vectors} of the real roots of the polynomials of the family~$P$ (in the case of two conjugate pairs, we do not specify whether these pairs are distinct or not). The notation which we use for the strata should be clear from the following example:

\begin{ex}
{\rm There is a single stratum $\mathcal{T}_5$ corresponding to a polynomial with a five-fold real root. This polynomial is}

\begin{equation}\label{T5}
(x+1/5)^5=x^5+x^4+2x^3/5+2x^2/25+x/125+1/3125
\end{equation}
{\rm and the stratum $\mathcal{T}_5$ is of dimension $0$ in $\mathbb{R}^4$. The strata $\mathcal{T}_{4,1}$, $\mathcal{T}_{3,2}$, $\mathcal{T}_{2,3}$ and $\mathcal{T}_{1,4}$ are of dimension 1 in $\mathbb{R}^4$; they correspond to polynomials of the form $(x-x_1)^m(x-x_2)^{5-m}$, where $x_1<x_2$ and $m=4$, $3$, $2$ and $1$ respectively; hence $mx_1+(5-m)x_2=-1$. The stratum $\mathcal{T}_3$ is of dimension $2$ and corresponds to polynomials $(x-x_1)^3(x^2+ux+v)$, where $u^2<4v$.}

\end{ex}

\begin{rems}\label{dimtang} 

{\rm (1) The dimension of a stratum is equal to the number of distinct roots (real or complex) minus 1; we subtract 1, because the sum of all roots equals $(-1)$. Thus $\mathcal{T}_{4,1}$, $\mathcal{T}_{3,2}$, $\mathcal{T}_{2,3}$ and 
$\mathcal{T}_{1,4}$ are the only strata of dimension 1. As $d=5$, i.e. as $d$ is odd, there is always at least one real root, so a stratum corresponding to polynomials having at least one conjugate pair (hence to polynomials having $\geq 3$ 
distinct roots) is of dimension $\geq 2$. This is the case of the stratum~$\mathcal{T}_3$.

(2) The tangent space at any point of any stratum of dimension 1, 2 or 3 is transversal to the space $Obcd$, $Ocd$ or $Od$
respectively. This follows from~\cite[Theorem~2]{VMJ}}.
\end{rems}

\noindent On Fig.~\ref{fig1} and Fig.~\ref{fig2} we show the projections in the $(a,b)$-plane of the strata $\mathcal{T}_5$, $\mathcal{T}_{4,1}$, $\mathcal{T}_{3,2}$, $\mathcal{T}_{2,3}$ and $\mathcal{T}_{1,4}$. The union of the projections of the three strata $\mathcal{T}_{4,1}$, $\mathcal{T}_5$ and $\mathcal{T}_{1,4}$ (resp. $\mathcal{T}_{3,2}$, $\mathcal{T}_5$ and $\mathcal{T}_{2,3}$) is an algebraic curve drawn by a solid (resp. dashed) line and having a cusp at the projection of $\mathcal{T}_5$; the coordinates of the projection of $\mathcal{T}_5$ are (2/5, 2/25), see (\ref{T5}). When following a vertical line (i.e. parallel to the $b$-axis) from below to above, the projections of the strata are intersected in the following order: $\mathcal{T}_{4,1}$, $\mathcal{T}_{3,2}$, $\mathcal{T}_{2,3}$,  $\mathcal{T}_{1,4}$. These projections and the $a$- and $b$-axes define 15 open {\em zones} in  $\mathbb{R}^2$ (the space $Oab$), denoted by $A$, $B$, $\ldots$, $M$, $N$ and $P$.

The SPs which we use begin with $(+,+)$. In the right upper corner of Fig.~\ref{fig1} the notation $\sigma=(+,+,+,+,\tilde{\sigma})$ means that when one chooses the values of the variables $(a,b)$ from the first quadrant, then this defines the SP $\sigma$, in which $\tilde{\sigma}$ stands for the couple of signs of the variables $(c,d)$ (and similarly for the other three corners of Fig.~\ref{fig1}). Recall that in the plane, the four open quadrants correspond to the following couples of signs of the two coordinates: 

$$ {\rm I}:~(+,+),~{\rm II}:~(-,+),~{\rm III}:~(-,-)~{\rm and}~{\rm IV}:~(+,-)~.$$ 

\begin{nota}\label{notasigmaij}
{\rm Further in the text, we use the following notation: $\sigma _{i,j}$ means that the signs of the variables $(a,b)$ correspond to the $i$th and the ones of the variables $(c,d)$ to the $j$th quadrant. Thus the SPs $(+,+,-,+,+,-)$ and $(+,+,+,+,-,-)$ are denoted by $\sigma _{2,4}$ and $\sigma _{1,3}$ respectively.

We explain now the meaning of the pictures. On Fig.~\ref{fig1},~\ref{fig2}, ~\ref{projstrata1} and~\ref{projstrata2} we represent the plane $(a,b)$; the $a$-axis is horizontal and the $b$-axis is vertical. On the rest of the figures we represent the plane $(c,d)$; the $c$-axis is horizontal and the $d$-axis is vertical. We fix a value $(a_0,b_0)$ of the couple $(a,b)$ from one of the domains $A$, $\ldots$, $N$, $P$, and we draw the set $\Delta ^{\sharp}:=\Delta |_{(a,b)=(a_0,b_0)}$. The figures thus obtained resemble the ones given in \cite{Pos} in relationship with the butterfly catastrophe. Indeed, in the definition of the latter one uses a degree 5 monic polynomial family $S$ with vanishing coefficient of $x^4$. The family $P(x,a,b,c,d)$, see (\ref{eqP}), is obtained from $S$ via the shift $x\mapsto x+1/5$ which means making an upper-triangular affine transformation in the space of coefficients. The convexity of the curves shown on the figures results from the following theorem of I. M\'eguerditchian, see~\cite[Proposition~1.3.3]{Me}, which is a generalization of a result of B. Chevallier, see~\cite{Che}.}
\end{nota}

\begin{tm}\label{tmMe}
Locally the discriminant set $\Delta$ at a point, where it is smooth, belongs entirely to one of the two half-spaces defined by its tangent hyperplane, namely, the one, where the polynomial $P$ has two more real roots.
\end{tm}

We use also another result of \cite{Me}: 

\begin{lm}\label{lmproduct}

{\rm [Lemma about the product]} Suppose that $P_1$, $\ldots$, $P_s$ are monic polynomials, where for $i\neq j$, the polynomials $P_i$ and $P_j$ have no root in common. Set $P:=P_1\cdots P_s$. Then there exist open neighbourhoods $U_k$ of $P_k$ and $U$ of $P$ such that for $Q_k\in U_k$, the mapping 
$$U_1\times \cdots \times U_s\rightarrow U~~~\, \, ,~~~\, \, (Q_1,\ldots ,Q_s)\mapsto Q_1\cdots Q_s$$ 
is a diffeomorphism. 
\end{lm} 


\begin{rems}\label{remsfighyp}
{\rm (1) From Lemma~\ref{lmproduct} one can deduce what local singularities of the sets $\Delta ^{\sharp}$ can be encountered. At a point $(a,b,c,d)$ for which the polynomial $\tilde{P}:=P(x,a,b,c,d)$ has one double and one or three simple roots the set $\Delta ^{\sharp}$ is smooth. Indeed, according to Lemma~\ref{lmproduct} in this case the set $\Delta ^{\sharp}$ is locally diffeomorphic to the cartesian product of the discriminant set of the family $x^2+ux+v$, $u$, $v\in \mathbb{R}$ (which is the curve $u^2=4v$), and $\mathbb{R}^2$.

At a point where $\tilde{P}$ has a triple real root and $2$ or $0$ simple real roots, the set $\Delta ^{\sharp}$ is diffeomorphic to the cartesian product of a semi-cubic parabola (i.e. a cusp) and $\mathbb{R}^2$. Indeed, the discriminant set of the family of polynomials $S_1:=x^3+ux+v$, $u$, $v\in \mathbb{R}$, is the curve $27v^2+4u^3=0$. The family $S_2:=x^3+wx^2+u^*x+v^*$ is obtained from $S_1$ via the shift $x\mapsto x+w/3$; here $u^*=u+w^2/3$ and $v^*=v+uw/3+w^3/27$. On the figures cusp points are denoted by $\kappa$, $\lambda$ and $\mu$. 

At a point where $\tilde{P}$ has one simple and two double roots, the set $\Delta ^{\sharp}$ is locally diffeomorphic to the cartesian product of two transversally intersecting smooth curves and $\mathbb{R}^2$, see Lemma~\ref{lmproduct}. On the figures, such points are denoted by $\phi$, $\psi$ or~$\theta$.     

At a point where $\tilde{P}$ has a triple and a double real roots, the set $\Delta ^{\sharp}$ is locally diffeomorphic to a cartesian product of $\mathbb{R}^2$ and the union of a semi-cubic parabola and a smooth arc passing through the cusp point and transversal to the geometric tangent at the cusp point. Such points belong to the 
strata $\mathcal{T}_{2,3}$ and $\mathcal{T}_{3,2}$. We do not show such sets $\Delta ^{\sharp}$ on the pictures.

Finally, if $\tilde{P}$ has a quadruple and a simple real roots (such points belong to the strata $\mathcal{T}_{1,4}$ and $\mathcal{T}_{4,1}$), then locally the set $\Delta ^{\sharp}$ is diffeomorphic to the cartesian product of a swallowtail and $\mathbb{R}$. For a picture of a swallowtail see~\cite{Pos}.

On the figures the letters $\alpha$ and $\omega$ denote the "infinite branches" of the sets $\Delta^{\sharp}$. 

There are no vertical tangent lines at any point of any of the sets $\Delta ^{\sharp}$, see part (2) of Remarks~\ref{dimtang}.

(2) On the figures the hyperbolicity domain (denoted by $h$) is represented by the following curvilinear triangles or quadrigons: 

\begin{tabular}{lllllll}
$\lambda \mu \phi$ &Figures &\ref{fig4} (right),& \ref{fig5} (left), &\ref{fig14} &and \ref{fig15} (right)&; \\
$\lambda \theta \psi \phi$ &Figures &\ref{fig9} (left), &\ref{fig6} (left), &\ref{fig13} &and \ref{fig16} (left) &; \\
$\lambda \theta \kappa$ &Figures &\ref{fig7},& \ref{fig12}, &\ref{fig17} (left)~&.  
\end{tabular}
  }
\end{rems}

How the set $\Delta$ looks like near the origin is justified by the following lemma:

\begin{lm}
In the space of the parameters $(a,b,c,d)$, a point with $c=d=0$ belongs to the discriminant set $\Delta$. If $b\neq 0=c=d$, then the set $\Delta$ is tangent to the hyperplane $d=0$. For $b>0$ and $b<0$ the set belongs locally to the half-plane $d\geq 0$ and $d\leq 0$ respectively.
\end{lm}

\begin{proof}
For $c=d=0$, the number $0$ is a double root of $P$ which proves the first claim of the lemma. For $b\neq 0=c=d$, the polynomial $P$ is locally representable in the form $(x+\varepsilon )^2(g+hx+ux^2+x^3)$,where $g\neq 0$. Hence $c=2\varepsilon g+\varepsilon ^2h$ and $d=g\varepsilon ^2$, i.e. the set $\Delta$ is locally defined by an equation of the form $d=c^2/4g+o(c^2)$ which proves the second claim. For $b>0$ and $b<0$ one has $g>0$ and $g<0$ respectively which proves the last statement of the lemma.
\end{proof}

\section{Pictures representing the discriminant set \protect\label{secpictures}}

On Figures \ref{fig3}-\ref{fig10} we show the set $\Delta ^{\sharp}$ for values of $(a,b)$ from the different zones shown on Figures 1 and 2. After each figure we indicate which cases are realizable in the domains $h$, $t$ and $s$, see Notation~\ref{notasigmaij}. Whenever a figure consists of two pictures, the one on the right is a detailed picture close to the origin. Under each of Fig.~\ref{fig3}-\ref{fig10} we indicate the cases which are realizable in the different parts of the plane $(c,d)$ delimited by the coordinate axes and the corresponding set $\Delta ^{\sharp}$. E.g. when after Fig. 4 under "domain t" we write "5 , 9 $\sigma _{2,1}$ (2,1) , (0,3)", this means that in the two parts of the domain $t$ in the first quadrant the cases~($\sigma _{2,1}$,(2,1)) and ($\sigma _{2,1}$, (0,3)) are realizable. The numbers 5 and 9 are numbers of different cases. These numbers are attributed in the order of appearance of the cases. When one and the same case appears in different zones, then it bears the same number. There are two figures corresponding to zone E, see the lines following 
Fig.~\ref{fig6}.

\begin{rems}\label{remscase}
{\rm (1) The following four rules hold true. They allow to define by 
continuity
the case (SP, AP) which is realizable in any domain of the $(c,d)$-plane 
for
$(a,b)$ fixed.

{\em i)} When the $c$-axis is crossed at a point not belonging to the set $\Delta ^{\sharp}$ and different from the origin, then exactly one real root changes sign. When a hyperplane $a=0$, $b=0$ or $c=0$ is crossed, then only the corresponding sign in the SP changes.

{\em ii)} On all pictures, in the part of the domain $s$ which is above the $c$-axis, the only real root of the polynomial $P$ is 
negative. Indeed, for $a$, $b$ and $c$ fixed and $d>0$ large enough, the polynomial $P$ has only one real root which is simple and negative.

{\em iii)} At a cusp point belonging to the closure of the domain $t$ (but not $h$) the triple root of $P$ has the same sign as the single real root in the adjacent $s$-domain.

{\em iv)} Suppose that a point follows the arc of the set $\Delta ^{\sharp}$ which passes through the point $(0,0)$ in the plane $(c,d)$. Then when the point passes from the first into the second quadrant or vice versa, a double real root of $P$ changes sign.

{\em v)} The AP corresponding to a point of the set~(\ref{eqRDelta}) and belonging to the domain $h$ is the Descartes' pair which is defined by the SP. This allows, for each of the Figures  \ref{fig3}-\ref{fig10} containing the domain $h$, to find the cases realizable in each of the parts of the domain $h$. 

{\em vi)} At a self-intersection point $\phi$ or $\psi$ one of the open sectors defined by the two intersecting arcs of $\Delta ^{\sharp}$ belongs to the domain $s$ and its opposite sector belongs to the domain $h$. The other two sectors (denoted here by $S_1$ and $S_2$) belong to the domain $t$. When a point moves from $\phi$ or $\psi$ into $S_1$, then
one of the two double roots of $P$ gives birth to a complex conjugate pair of roots. When a point moves from
$\phi$ or $\psi$ into $S_2$, then the other one of the two double roots of $P$ gives birth to such a pair. This 
follows from Lemma~\ref{lmproduct}.

(2) 
To the possible Descartes' pairs of the SPs beginning with $(+,+)$, i.e. to $(0,5)$, $(1,4)$, $(2,3)$, $(3,2)$ and $(4,1)$, there correspond 3, 3, 4, 4 and 3 possible APs respectively. E. g. to the Descartes' pair $(2,3)$ there correspond the possible APs $(2,3)$, $(0,3)$, $(2,1)$ and 
$(0,1)$. This means that for the four quadrants in the plane $(a,b)$ one obtains the following numbers of a priori possible couples (SP, AP) (we list also the zones of each quadrant on Fig.~\ref{fig1}-\ref{fig2} and the possible couples (SP, Descartes' pair) for the given quadrant):

  $$\begin{array}{llll}
{\rm I}:&13&L,M,N,P&(\sigma _{1,1},(0,5)), (\sigma _{1,2},(2,3)), (\sigma _{1,3},(1,4)), (\sigma _{1,4},(1,4))~,\\ \\
{\rm II}:&15&A,B,C&(\sigma _{2,1},(2,3)), (\sigma _{2,2},(4,1)), (\sigma _{2,3},(3,2)), (\sigma _{2,4},(3,2))~,\\ \\
{\rm III}:&15&D,E,F,G&(\sigma _{3,1},(2,3)), (\sigma _{3,2},(2,3)), (\sigma _{3,3},(1,4)), (\sigma _{3,4},(3,2))~,\\ \\
{\rm IV}:&15&H,I,J,K&(\sigma _{4,1},(2,3)), (\sigma _{4,2},(2,3)), (\sigma _{4,3},(1,4)), (\sigma _{4,4},(3,2))~.\end{array}$$

The non-realizable couple (the second of couples (\ref{eqGrAlFu})) 
corresponds to quadrant~II.

(3) On Fig. \ref{projstrata1}-\ref{projstrata2} we show the projections in the plane $(a,b)$ of the 
strata $\mathcal{T}_5$ and $\mathcal{T}_{i,j}$, $1\leq i,j\leq 4$, 
$i+j=5$, (see Fig.~\ref{fig1}-\ref{fig2}), and
of the set $\mathcal{M}$ of values of $(a,b)$ for which the set $\Delta 
^{\sharp}$ has a self-intersection at $c=d=0$. This is the set for which 
the polynomial $x^3+x^2+ax+b$ has a multiple root. It is drawn by a 
dotted line. It

(a) has a cusp point $(1/3,1/27)$ situated on the projection of the stratum $\mathcal{T}_{3,2}$;

(b) is tangent to the projection of $\mathcal{T}_{2,3}$ at $(a,b)=(1/4,0)$;

(c) is tangent to the projection of $\mathcal{T}_{1,4}$ at $(a,b)=(0,0)$;

(d) intersects the projection of $\mathcal{T}_{3,2}$ and $\mathcal{T}_{2,3}$ at $(a,b)=$
$$\begin{array}{ccc}(-8-4\sqrt{10})/15, (-152(2+\sqrt{10})+52)/675)&=&(-1.37\ldots , -1.39\ldots )~~~{\rm and}\\ 
\\
(-8+4\sqrt{10})/15, (-152(2-\sqrt{10})+52)/675)&=&(0.30\ldots , 
0.03\ldots )\end{array}$$
respectively;

(e) has no point in common with the projection of the stratum 
$\mathcal{T}_{4,1}$.

\noindent At these points the polynomial $P$ and the set $\Delta ^{\sharp}$ have respectively

(a) $P$: a triple negative root and a double root at $0$, $\Delta^{\sharp}$: a cusp point at $c=d=0$ with a non-horizontal tangent line at it and a smooth arc with a horizontal tangent line at $c=d=0$;

(b) $P$: a double negative root and a triple root at $0$, $\Delta ^{\sharp}$: a cusp at $c=d=0$ with a horizontal tangent line at it and a smooth arc with a non-horizontal tangent line at $c=d=0$;

(c) $P$: a negative a simple root and a quadruple root at $0$, $\Delta ^{\sharp}$: a $4/3$-singularity with a horizontal tangent line at $c=d=0$;

(d) at $\mathcal{T}_{3,2}$: $P$: a simple positive, a double negative and a double $0$ root, $\Delta ^{\sharp}$: two transversally intersecting arcs at $(0; 0)$ one of which with a horizontal tangent line; at 
$\mathcal{T}_{2,3}$: $P$: a simple and a double negative and a double $0$ root, $\Delta ^{\sharp}$: two transversally intersecting arcs at $(0; 0)$ one of which with a horizontal tangent line.

We do not include the set $\mathcal{M}$ in the partition of the plane $(a,b)$ into zones in order to keep the number of figures to be drawn reasonably low. Some changes of the relative position of the cusps of the set $\Delta ^{\sharp}$ and the coordinate axes $c$ and $d$ as the values of $a$ and $b$ change are commented between the figures.

(4) Two SPs, one corresponding to the third and one to the fourth quadrant in the $(a,b)$-plane, begin by $(+,+,-,-)$ and $(+,+,+,-)$ respectively. Hence if their last two signs are the same, then they contain one and the same number of sign changes and sign preservations. This means that one and the same APs correspond to them. Therefore the two couples (SP, AP), $(\sigma _{3,j},(k_1,k_2))$ and $(\sigma _{4,j},(k_1,k_2))$, are simultaneously realizable.}
\end{rems}

\begin{proof}[Proof of Proposition~\ref{prophowmany}]

We use Notation~\ref{notasigmaij} and the definition of the generators $g_1$ and $g_2$ of the $\mathbb{Z}_2\times \mathbb{Z}_2$-action, see Remark~\ref{remgen}. We consider only SPs beginning with $(+,+)$ which means that we deal with halves of orbits (the other halves are with SPs beginning with $(+,-)$, see Remark~\ref{remgen}). We consider the action of $g_1$ and $g_2$ not only on couples (SP, AP), but also just on SPs. Thus

\begin{equation}\label{eqone}
g_2(\sigma _{2,1})=\sigma _{4,1}~~~~{\rm and}~~~~g_2(\sigma _{1,3})=\sigma _{3,3}~.
\end{equation}

The Descartes' pair corresponding to $\sigma _{2,1}$ (resp. $\sigma_{1,3}$) equals $(2,3)$ (resp. $(1,4)$), see part (2) of Remarks~\ref{remscase}. We denote by $\rho$ any of the APs, so $\rho =(2,3)$, $(0,3)$, $(2,1)$ or $(0,1)$ (resp. $\rho =(1,4)$, $(1,2)$ or $(1,0)$). Thus ($(\sigma _{2,1},\rho )$, $(\sigma _{4,1},\rho )$) (resp. ($(\sigma _{1,3},\rho )$, $(\sigma _{3,3},\rho )$)) are half-orbits; the other halves are of the form ($g_1((\sigma _{2,1},\rho ))$, $g_1((\sigma _{4,1},\rho ))$) (resp. ($g_1((\sigma _{1,3},\rho ))$, $g_1((\sigma _{3,3},\rho ))$)), see Remark~\ref{remgen}. Thus we have described $4+3=7$ orbits of length~$4$.
Next, one has

\begin{equation}\label{eqtwo}
g_1g_2(\sigma _{1,2})=\sigma _{4,4}~,~g_1g_2(\sigma _{2,2})=\sigma _{1,4}~,~g_1g_2(\sigma _{3,2})=\sigma _{2,4}~~{\rm and}~~g_1g_2(\sigma _{4,2})=\sigma _{3,4}~.
\end{equation}

The half-orbits in the case of $\sigma _{1,2}$ are of the form ($(\sigma _{1,2},\rho )$, $(\sigma _{4,4},\rho ^R)$),
where $\rho ^R$ is obtained from $\rho$ by exchanging the two components, similarly for the other cases in (\ref{eqtwo}). Thus one obtains $4+3+4+4=15$ more orbits of length $4$, see part (2) of Remarks~\ref{remscase}.

Finally, one obtains
\begin{equation}\label{eqthree}
g_2(\sigma _{1,1})=\sigma _{1,1}~,~g_2(\sigma _{2,3})=\sigma _{2,3}~,~g_2(\sigma _{3,1})=\sigma _{3,1}~~{\rm and}~~g_2(\sigma _{4,3})=\sigma _{4,3}~,
\end{equation}
so by analogy with the SPs involved in (\ref{eqone}) one obtains half-orbits of length~$1$ hence orbits of
length~$2$. Their quantity is $3+4+4+3=14$. One of them is the only non-realizable orbit, see the second couple in~(\ref{eqGrAlFu}). There remains to notice that each possible SP $\sigma _{i,j}$ participates in exactly one of the equalities (\ref{eqone}), (\ref{eqtwo}) and (\ref{eqthree}), and to remind that the generators $g_1$ and $g_2$ are commuting involutions. Hence we have described all orbits of the $\mathbb{Z}_2\times \mathbb{Z}_2$-action.
\end{proof}

\begin{figure}[H]
\vskip0.5cm
\centerline{\hbox{\includegraphics[scale=0.27]{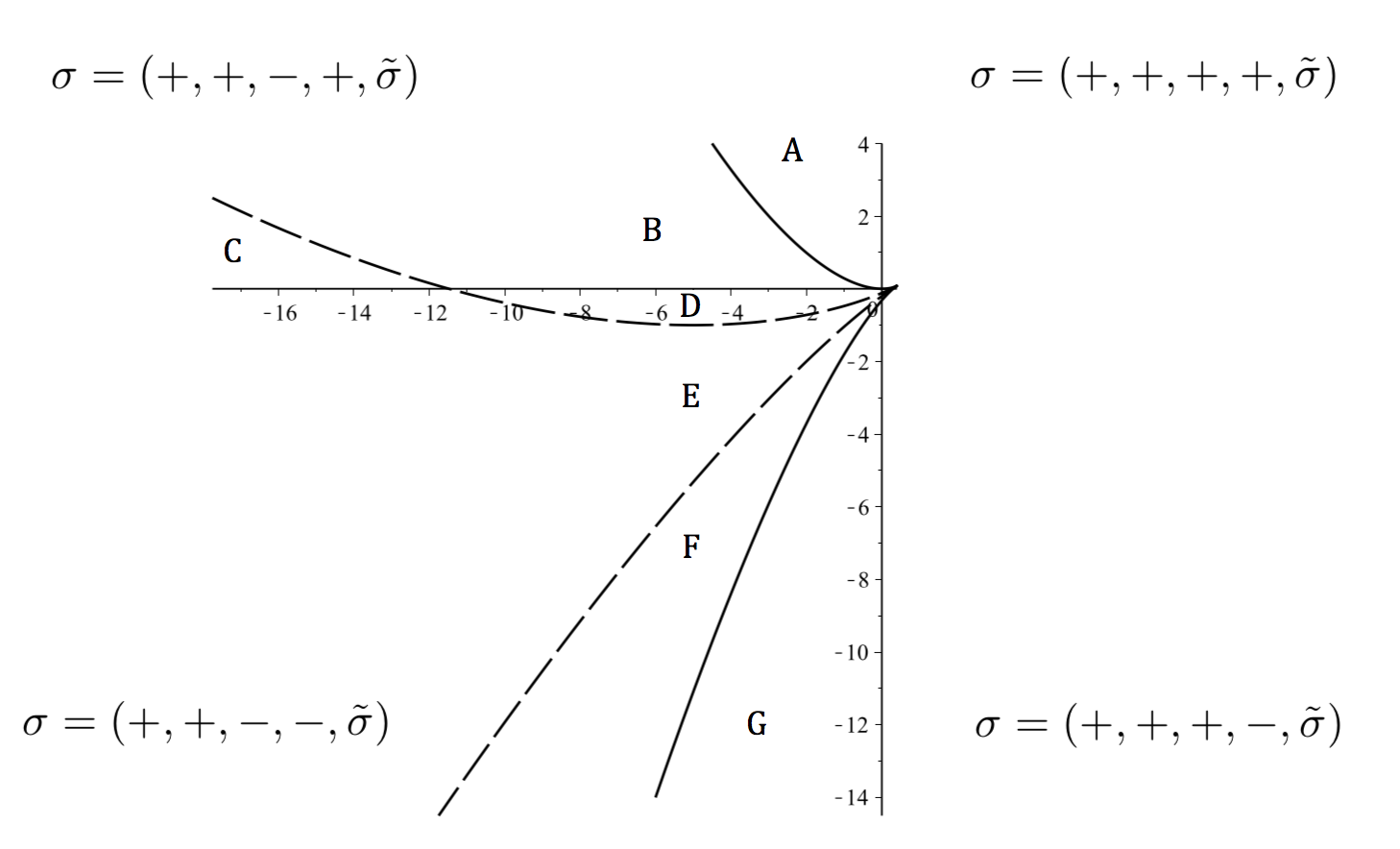}}}
    \caption{The projection of the discriminant locus of $P$ to the plane of the parameters $(a, b)$.}
\label{fig1}
\end{figure}

\begin{figure}[H]
\vskip0.5cm
\centerline{\hbox{\includegraphics[scale=0.25]{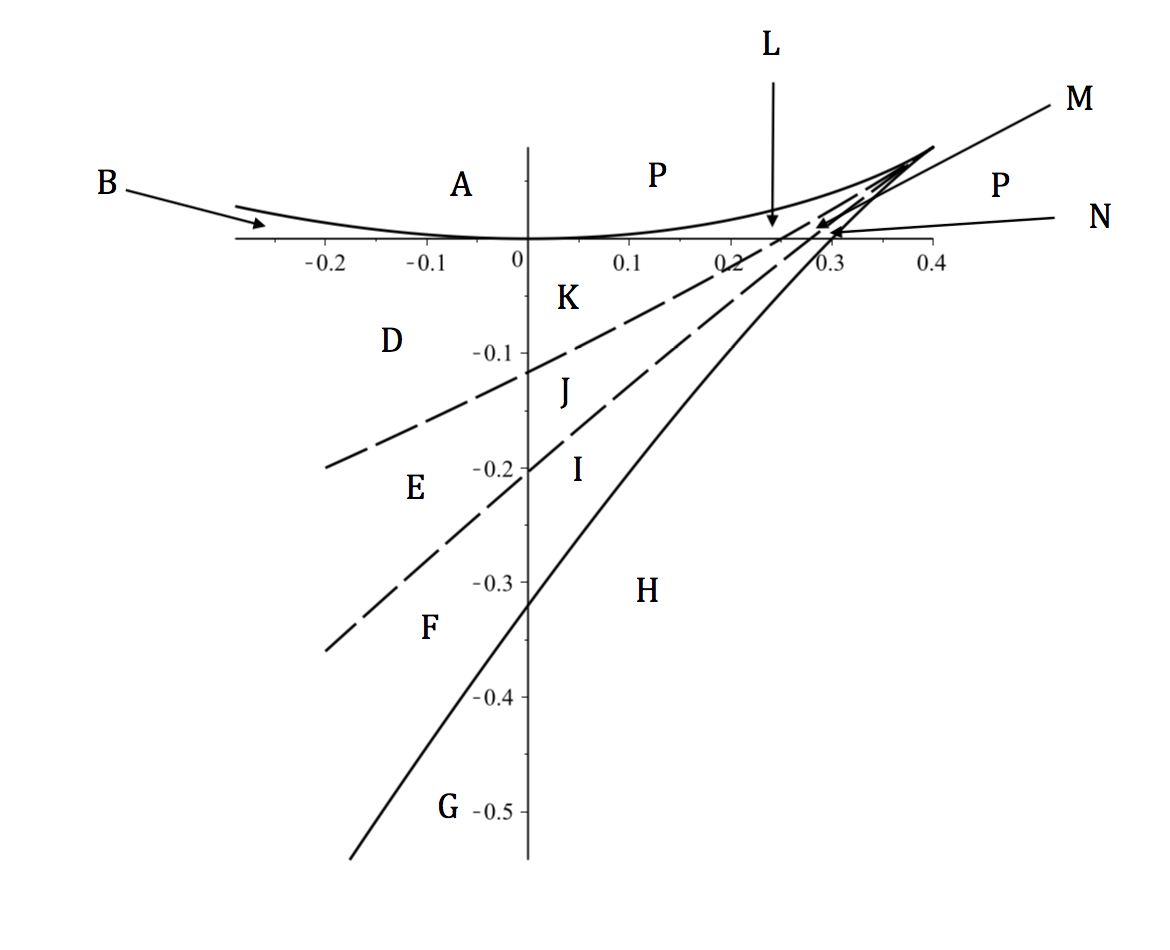}}}

    \caption{ Picture of the projection in the plane $(a,b)$ of the discriminant locus 
of $P$ with an enlarged
portion near the cusp point.}
\label{fig2}
\end{figure}

\begin{figure}[H]
\vskip0.5cm
\centerline{\hbox{\includegraphics[scale=0.32]{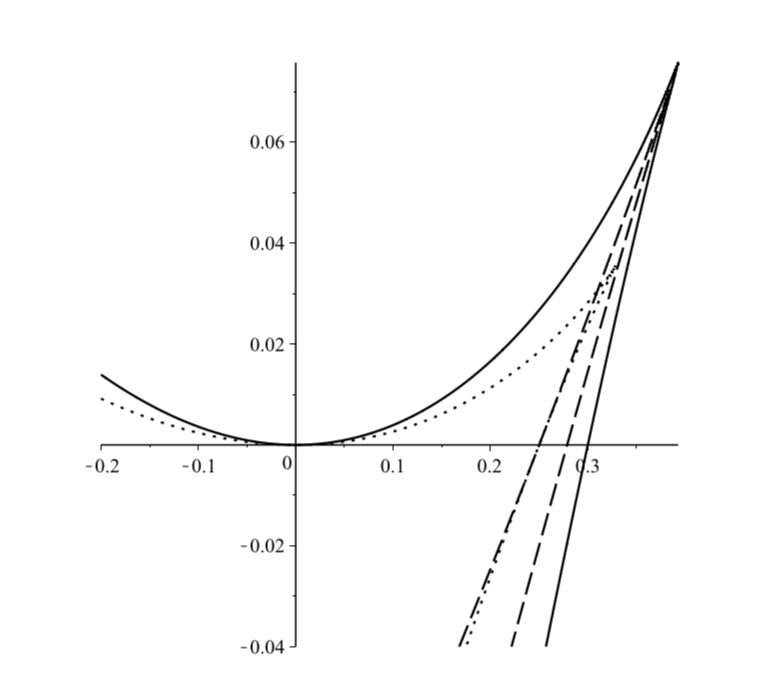}}}

    \caption{ The projections in the plane $(a,b)$ of the strata $\mathcal{T}_5$ and $\mathcal{T}_{i,j}$, $1\leq i,j\leq 4$, $i+j=5$, and of the set $\mathcal{M}$.}
\label{projstrata1}
\end{figure}

\begin{figure}[H]
\vskip0.5cm
\centerline{\hbox{\includegraphics[scale=0.33]{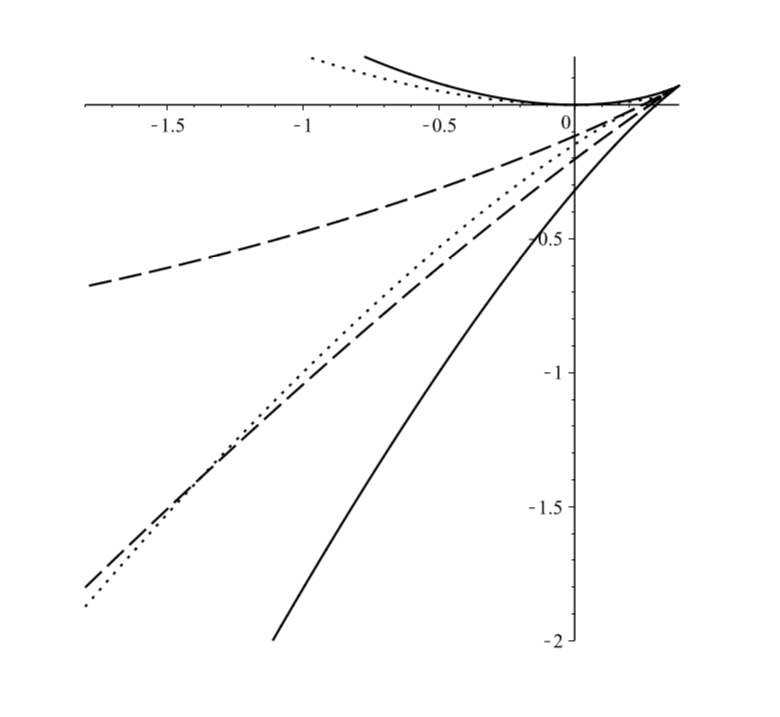}}}

    \caption{ The projections in the plane $(a,b)$ of the strata $\mathcal{T}_5$ and $\mathcal{T}_{i,j}$, $1\leq i,j\leq 4$, $i+j=5$, and of the set $\mathcal{M}$ (with enlarged portion near the cusp points).}
\label{projstrata2}
\end{figure}

  
\begin{figure}[H]
\vskip0.5cm
\centerline{\hbox{\includegraphics[scale=0.26]{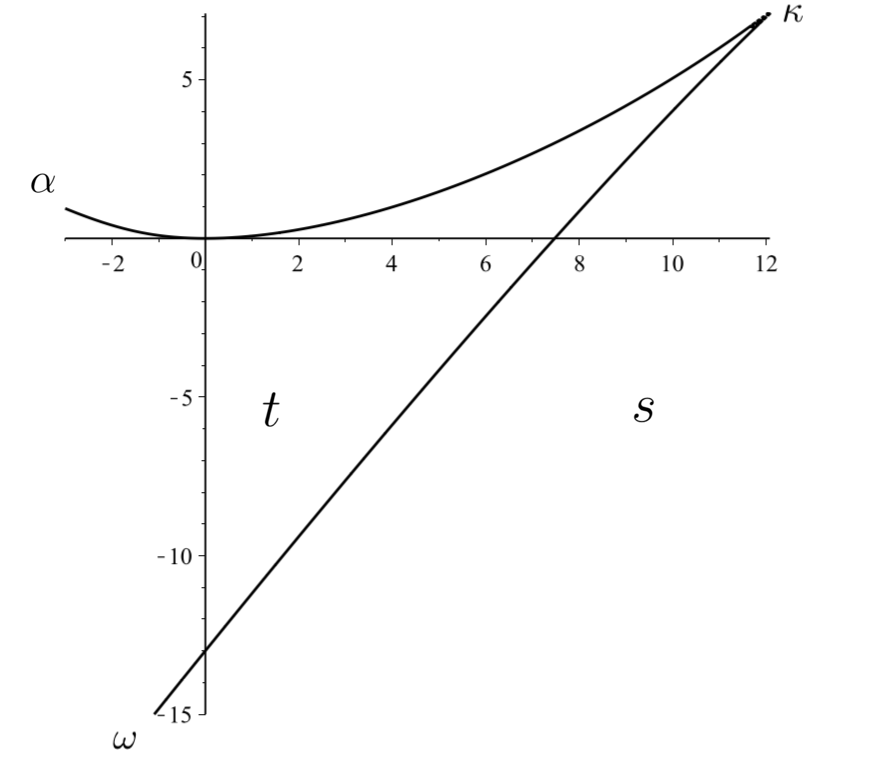}}}

    \caption{Zone A: The set $\Delta^{\sharp}$ for $a=-2$ and $b=3$.}
\label{fig3}

\end{figure}

\begin{tabular}{lclllc}
&domain $s$&&&domain $t$ \\ \\
1&$\sigma_{2,1}$ \, (0,1) &\qquad \qquad \qquad &5&$ \sigma_{2,1}$ \, (0,3)  \\
2&$\sigma_{2,2}$ \, (0,1)&\qquad \qquad \qquad &6&$\sigma_{2,2}$ \, (2,1)  \\
3&$\sigma_{2,3}$ \, (1,0)&\qquad \qquad \qquad &7&$\sigma_{2,3}$ \, (1,2) \\
4&$\sigma_{2,4}$ \, (1,0)  &\qquad \qquad \qquad &8&$\sigma_{2,4}$ \, (1,2)  \\

\end{tabular} \medskip \\

Cases 2 and 6 are realizable in the open second quadrant, above and below the set $\Delta ^{\sharp}$ respectively;
cases 3 and 7 in the open third quadrant, below and above $\Delta ^{\sharp}$ respectively; cases 4 and 8 in the open
fourth quadrant, below and above $\Delta ^{\sharp}$ respectively. Cases 1 and 5 occupy the open first quadrant minus
the set $\Delta ^{\sharp}$, case 5 the curvilinear triangle and case 1 the rest of the quadrant.


\begin{figure}[H]
\vskip0.5cm
\centerline{\hbox{\includegraphics[scale=0.3]{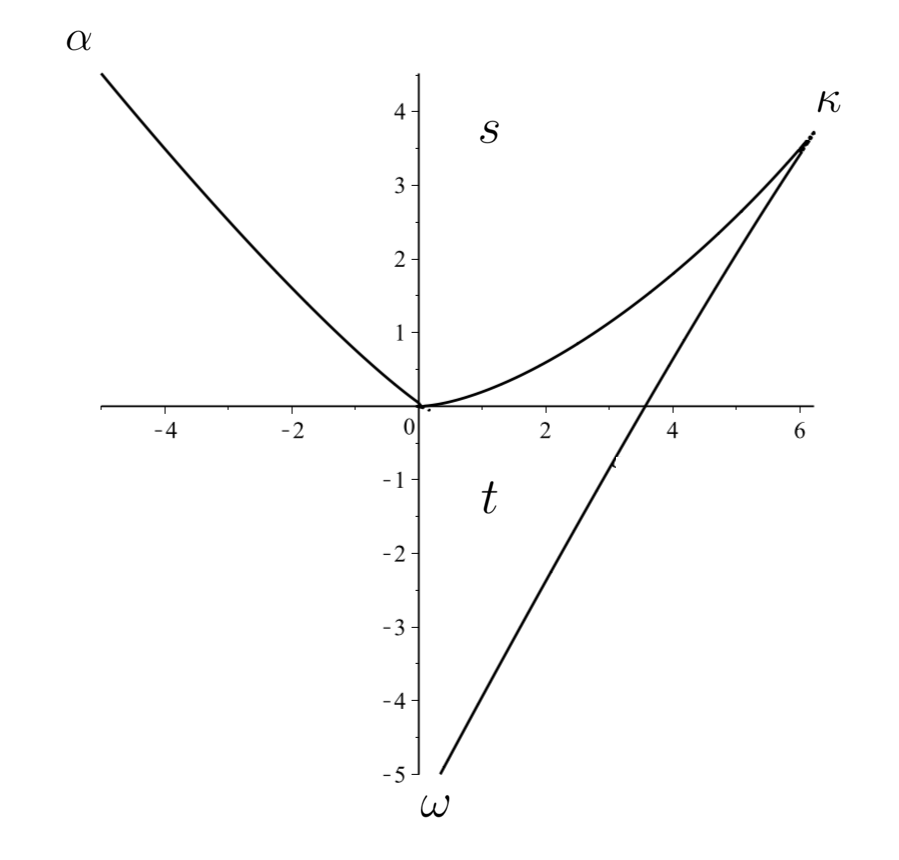}}\hskip1cm \hbox{\includegraphics[scale=0.3]{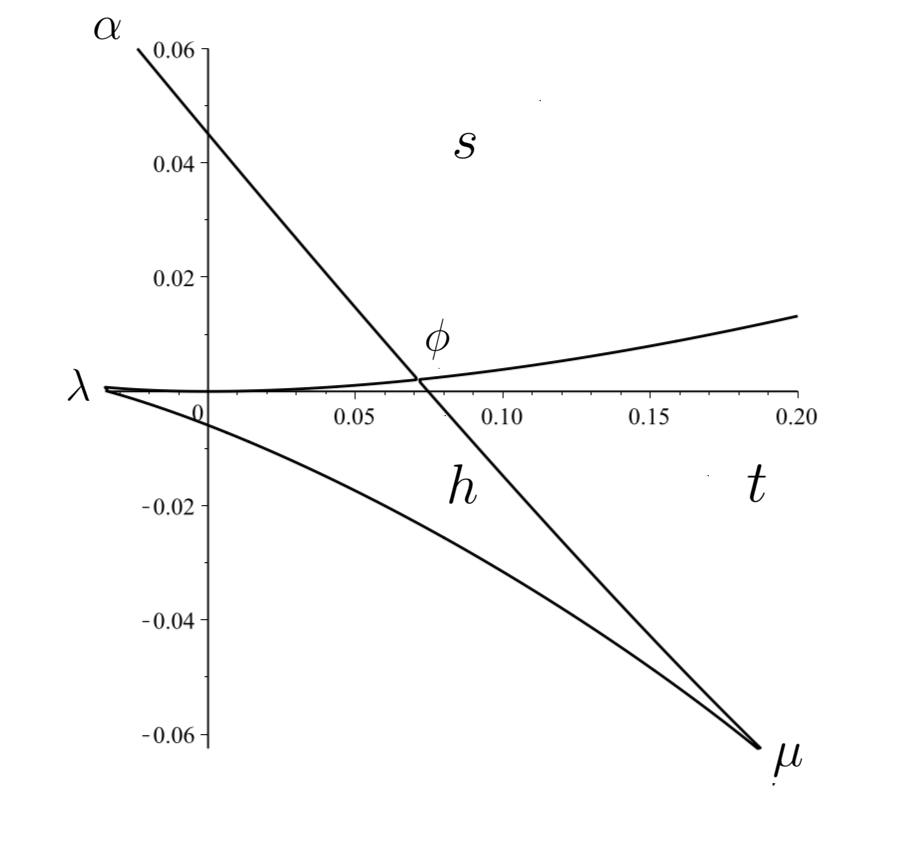}}}
    \caption{Zone B: The set $\Delta^{\sharp}$ for $a=-2$ and $b=0.5$.} 
    
\label{fig4}
\end{figure}

\begin{tabular}{lclllcllc}
&domain $s$&&&domain $t$&&&domain $h$ \\ \\

1&$\sigma_{2,1}$ \, (0,1)&&5 , 9&$ \sigma_{2,1}$ \, (0,3) , (2,1)  &\qquad \qquad \qquad&10&$\sigma_{2,1}$ \, (2,3)  \\
2& $\sigma_{2,2}$ \, (0,1)&\qquad \qquad \qquad&6&$\sigma_{2,2}$ \, (2,1)&&11&$\sigma_{2,2}$ \, (4,1)  \\
3&$\sigma_{2,3}$ \, (1,0) &\qquad \qquad \qquad&7&$\sigma_{2,3}$ \, (1,2) &&12&$\sigma_{2,3}$ \, (3,2)  \\
4&$\sigma_{2,4}$ \, (1,0) &\qquad \qquad \qquad&8&$\sigma_{2,4}$ \, (1,2)  &&13&$\sigma_{2,4}$ \, (3,2) \\ \\

\end{tabular} \medskip \\ 

The infinite branch $\omega$ intersects the negative $d$-half-axis.


\begin{figure}[H]
\vskip0.5cm
\centerline{\hbox{\includegraphics[scale=0.3]{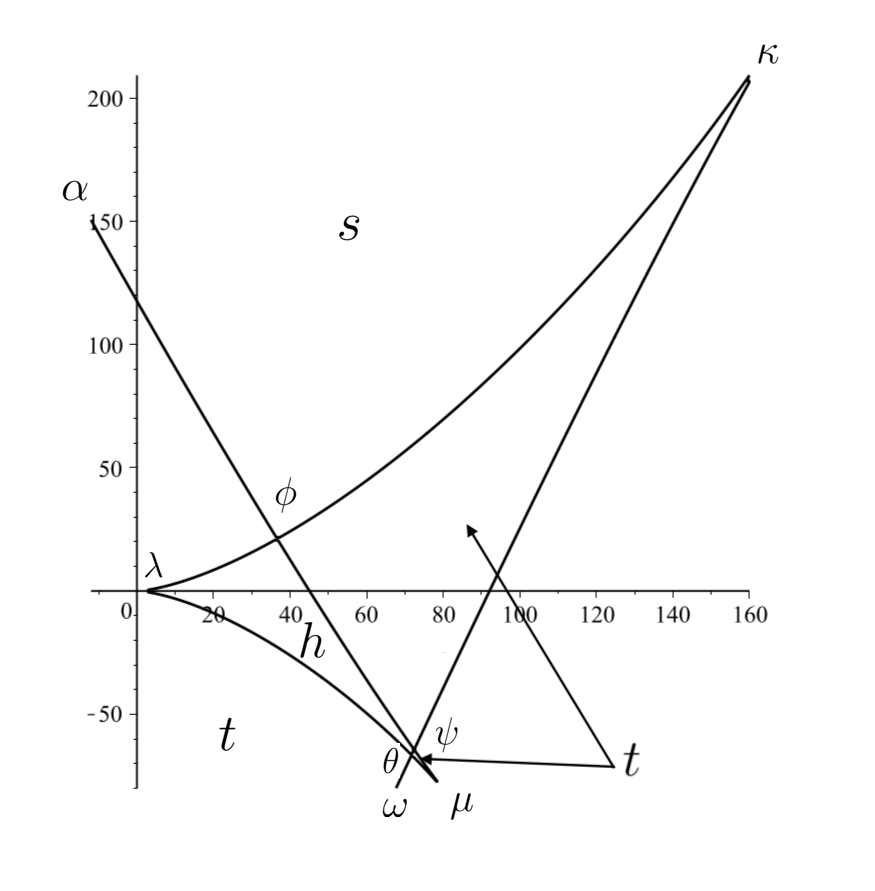}}\hskip1cm \hbox{\includegraphics[scale=0.42]{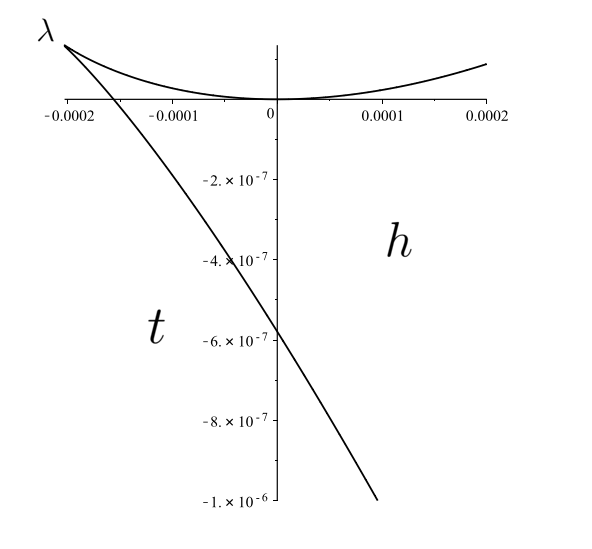}}}
    \caption{Zone C: The set $\Delta^{\sharp}$ for $a=-16$ and $b=0.1$.} 
    
    
\label{fig9}

\end{figure}

\begin{tabular}{lclllcllc}
&domain $s$&&&domain $t$&&&domain $h$ \\ \\

1&$\sigma_{2,1}$ \, (0,1) &&5 , 9&$ \sigma_{2,1}$ \,  (0,3) , (2,1) &\qquad \qquad \qquad&10&$\sigma_{2,1}$ \, (2,3)   \\
2&$\sigma_{2,2}$ \, (0,1)&\qquad \qquad \qquad& 6&$\sigma_{2,2}$ \, (2,1)&&11&$\sigma_{2,2}$ \, (4,1)  \\
3&$\sigma_{2,3}$ \, (1,0) &\qquad \qquad \qquad& 7&$\sigma_{2,3}$ \, (1,2)&&12&$\sigma_{2,3}$ \, (3,2)  \\
4&$\sigma_{2,4}$ \, (1,0)  &\qquad \qquad \qquad&8 , 14&$\sigma_{2,4}$ \, (1,2) , (3,0)&&13&$\sigma_{2,4}$ \, (3,2)   \\ \\

\end{tabular} \medskip

On Fig.~\ref{fig9}, the infinite branch $\omega$ intersects the negative $d$-half-axis. The intersection of the domain $t$ with the fourth quadrant consists of three curvilinear triangles. In the one which borders the third quadrant the AP is $(1,2)$ (as in the intersection of the domain $t$ with the third quadrant), in the one which belongs entirely to the interior of the fourth quadrant it is $(3,0)$ and in the one which borders the first quadrant it is again $(1,2)$.


\begin{figure}[H]
\vskip0.5cm
\centerline{\hbox{\includegraphics[scale=0.3]{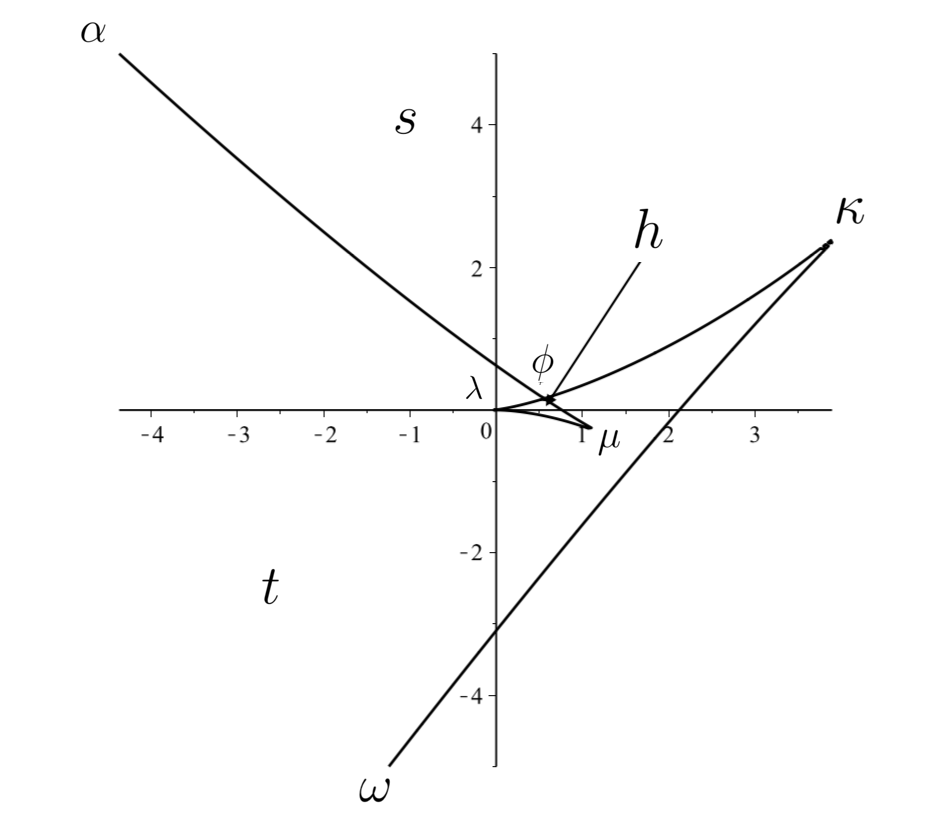}}\hskip1cm \hbox{\includegraphics[scale=0.3]{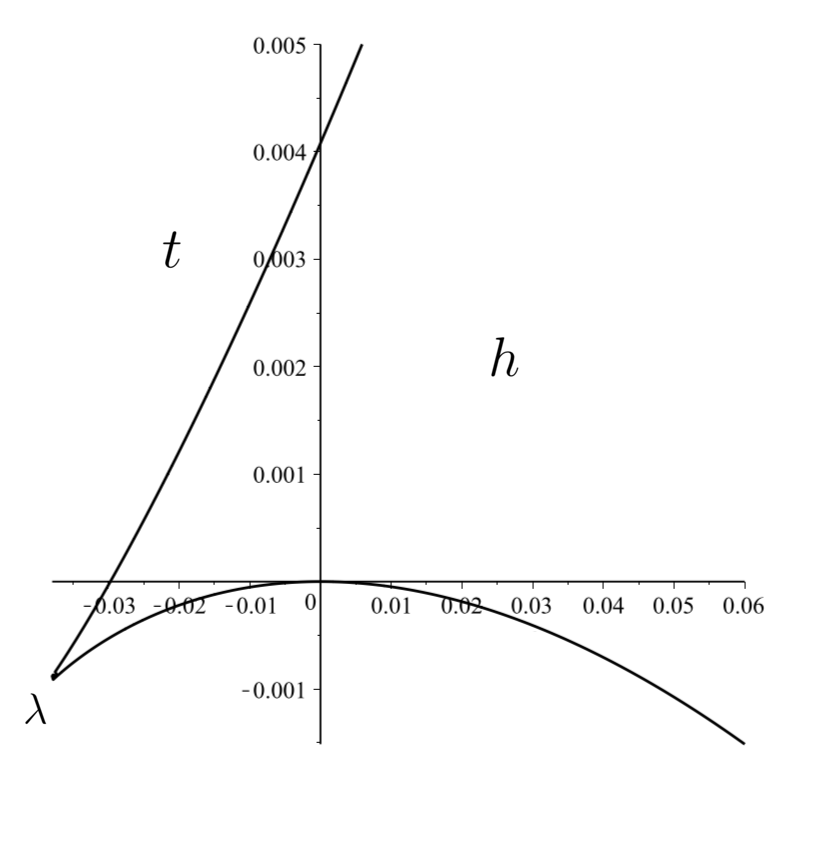}}}
    \caption{Zone D: The set $\Delta^{\sharp}$ for $a=-2$ and $b=-0.5$.}
    
\label{fig5}

\end{figure}

\begin{tabular}{lclllcllc}

&domain $s$&&&domain $t$&&&domain $h$ \\ \\

15&$\sigma_{3,1}$ \, (0,1)&&19 , 20&$ \sigma_{3,1}$ \, (0,3) , (2,1)&\qquad \qquad \qquad&24&$\sigma_{3,1}$ \, (2,3) \\
16&$\sigma_{3,2}$ \, (0,1) &\qquad \qquad \qquad&21&$\sigma_{3,2}$ \, (2,1) &&25&$\sigma_{3,2}$ \, (2,3)  \\
17&$\sigma_{3,3}$ \, (1,0) &\qquad \qquad \qquad&22&$\sigma_{3,3}$ \, (1,2) &&26&$\sigma_{3,3}$ \, (1,4)   \\
18&$\sigma_{3,4}$ \, (1,0) &\qquad \qquad \qquad&23& $\sigma_{3,4}$ \, (1,2)  &&27&$\sigma_{3,4}$ \, (3,2)  \\ \\

\end{tabular} \medskip


\begin{figure}[H]
\vskip0.5cm
\centerline{\hbox{\includegraphics[scale=0.3]{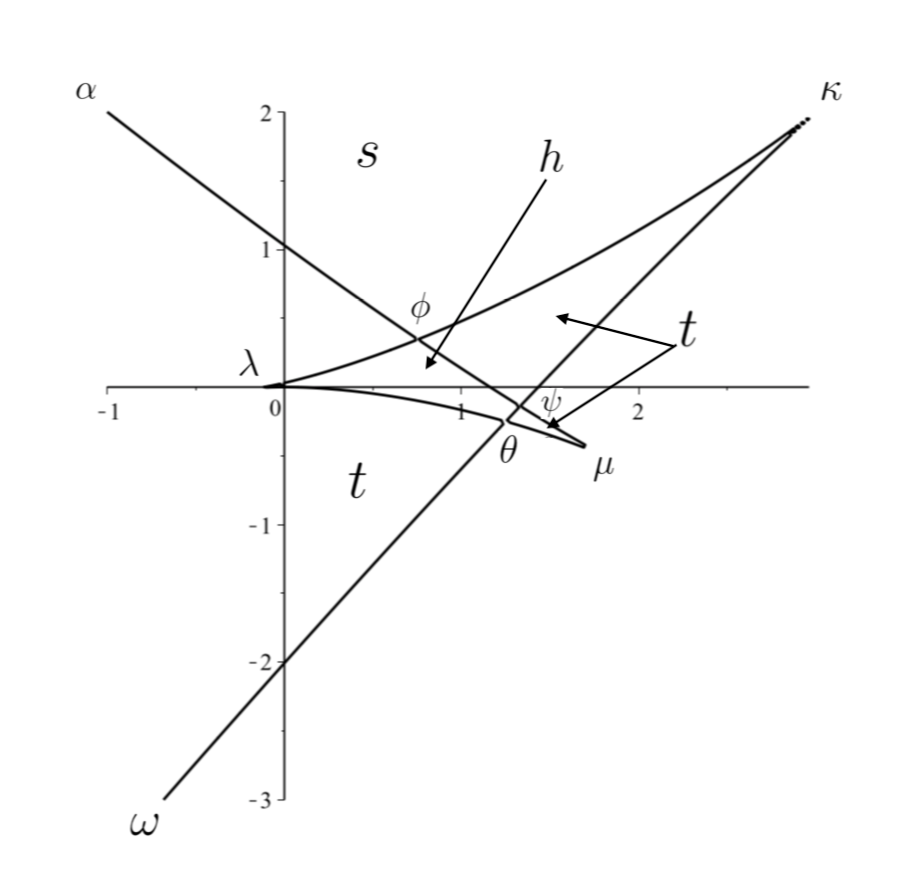}}\hskip1cm \hbox{\includegraphics[scale=0.44]{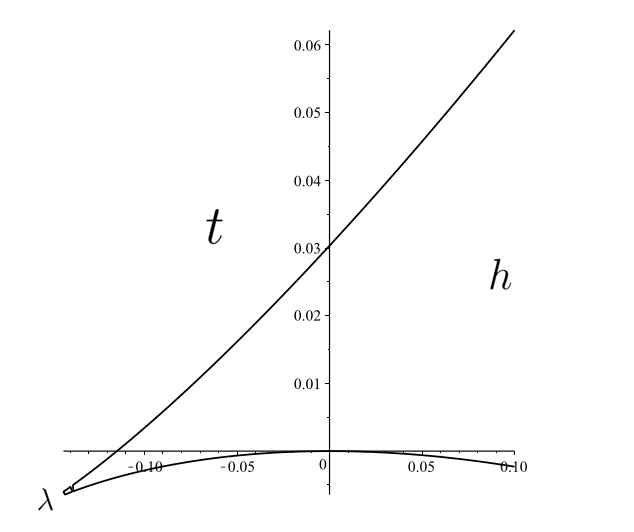}}}
    \caption{Zone E: The set $\Delta^{\sharp}$ for $a=-2$ and $b=-1$.} 
    
\label{fig6}
\end{figure}

\begin{tabular}{lclllcllc}
&domain $s$&&&domain $t$&&&domain $h$ \\ \\

15&$\sigma_{3,1}$ \, (0,1)&&19 , 20 & $ \sigma_{3,1}$ \, (0,3) , (2,1)  &\qquad \qquad \qquad&24&$\sigma_{3,1}$ \, (2,3)\\
16&$\sigma_{3,2}$ \, (0,1)&\qquad \qquad \qquad&21& $\sigma_{3,2}$ \, (2,1)&&25&$\sigma_{3,2}$ \, (2,3)   \\
17&$\sigma_{3,3}$ \, (1,0)&\qquad \qquad \qquad&22&$\sigma_{3,3}$ \, (1,2)&&26&$\sigma_{3,3}$ \, (1,4)    \\
18& $\sigma_{3,4}$ \, (1,0)&&23 , 28& $\sigma_{3,4}$ \, (1,2) , (3,0)&&27&$\sigma_{3,4}$ \, (3,2) \\ \\

\end{tabular} \medskip


On Fig.~\ref{fig6}  the self-intersection point $\phi$ is to the right while on Fig.~\ref{figEo}  it is to the
left of the $d$-axis. This is why case 29 is present only on the second of these figures. On Fig.~\ref{fig6} there are two domains corresponding to case 23. If one compares Fig.~\ref{fig6} with Fig.~\ref{fig5} one sees that for $a=-2$, as $b$ increases from $-1$ to $-0.5$, the two domains of case 23 fuse in one single domain.

\begin{figure}[H]
\vskip0.5cm
\centerline{\hbox{\includegraphics[scale=0.4]{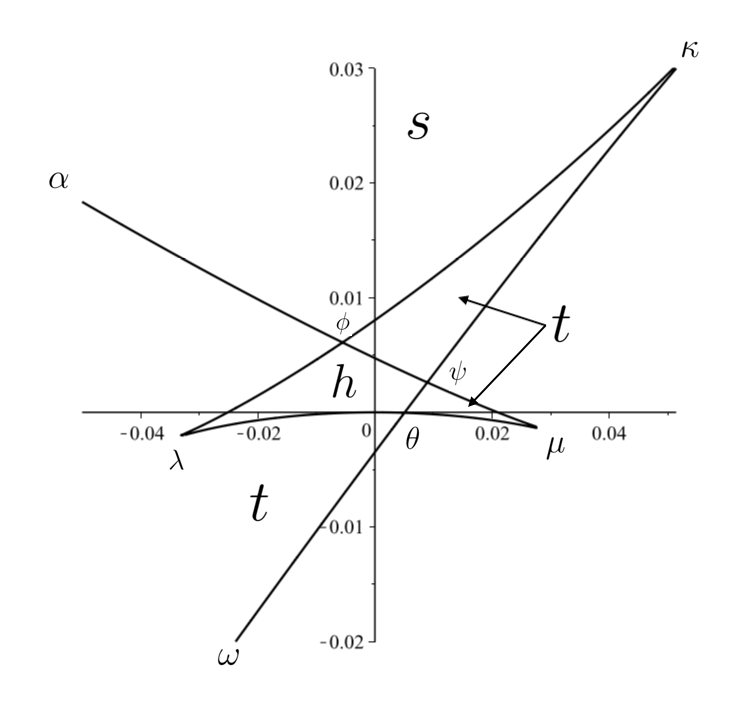}}}
    \caption{Zone E: The set $\Delta^{\sharp}$ for $a=-0.014$ and $b=-0.15$.} 
    
\label{figEo}
\end{figure}

\begin{tabular}{lclllcllc}
&domain $s$&&&domain $t$&&&domain $h$ \\ \\

15&$\sigma_{3,1}$ \, (0,1)&&19 , 20 & $ \sigma_{3,1}$ \, (0,3) , (2,1)  &\qquad \qquad \qquad&24&$\sigma_{3,1}$ \, (2,3)\\
16&$\sigma_{3,2}$ \, (0,1)&\qquad \qquad \qquad&21 , 29& $\sigma_{3,2}$ \, (2,1) , (0,3)&&25&$\sigma_{3,2}$ \, (2,3)   \\
17&$\sigma_{3,3}$ \, (1,0)&\qquad \qquad \qquad&22&$\sigma_{3,3}$ \, (1,2)&&26&$\sigma_{3,3}$ \, (1,4)    \\
18& $\sigma_{3,4}$ \, (1,0)&&23 , 28& $\sigma_{3,4}$ \, (1,2) , (3,0)&&27&$\sigma_{3,4}$ \, (3,2) \\ \\

\end{tabular} \medskip


\begin{figure}[H]
\vskip0.5cm
\centerline{\hbox{\includegraphics[scale=0.4]{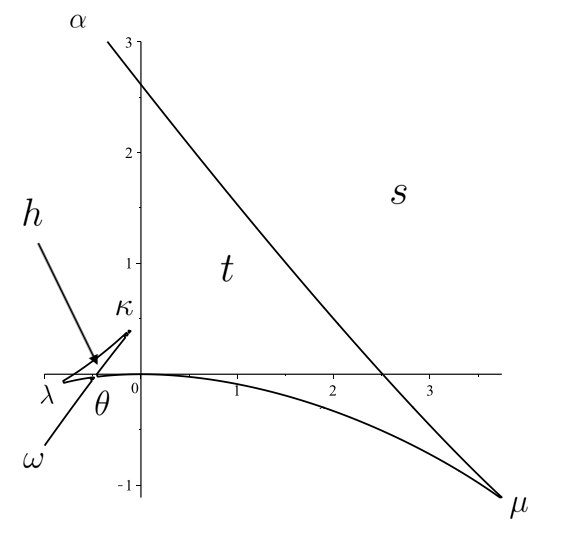}}\hskip1cm \hbox{\includegraphics[scale=0.4]{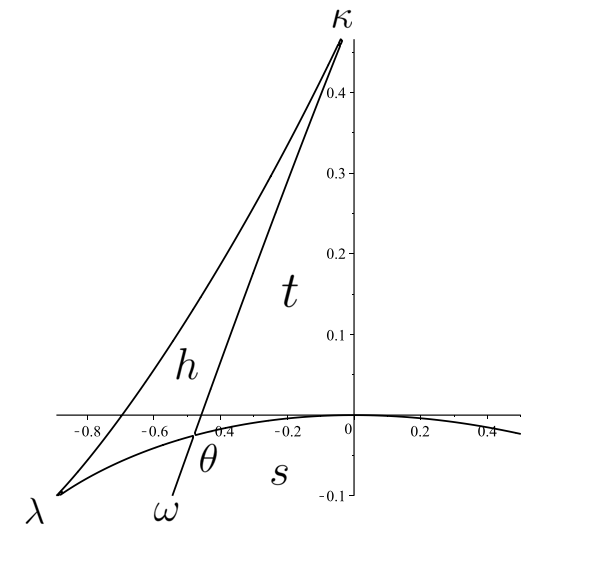}}}
    \caption{Zone F: The set $\Delta^{\sharp}$ for $a=-2$ and $b=-2.5$.} 
    
\label{fig7}
\end{figure}

\begin{tabular}{lclllcllc}
&domain $s$&&&domain $t$&&&domain $h$ \\ \\

15&$\sigma_{3,1}$ \, (0,1) && 20&$ \sigma_{3,1}$ \, (2,1)&\qquad \qquad \qquad& \\
16&$\sigma_{3,2}$ \, (0,1)&\qquad \qquad \qquad& 21&$\sigma_{3,2}$ \, (2,1) &&25&$\sigma_{3,2}$ \, (2,3)  \\
17& $\sigma_{3,3}$ \, (1,0) &\qquad \qquad \qquad&22&$\sigma_{3,3}$ \, (1,2)&&26&$\sigma_{3,3}$ \, (1,4) \\
18&$\sigma_{3,4}$ \, (1,0) &\qquad \qquad \qquad&28&$\sigma_{3,4}$ \, (3,0)&&   \\  \\

\end{tabular} \medskip

The intersection of the domain $t$ with the second quadrant consists of two parts. In both of them one and the same case is realizable.


\begin{figure}[H]
\vskip0.5cm
\centerline{\hbox{\includegraphics[scale=0.4]{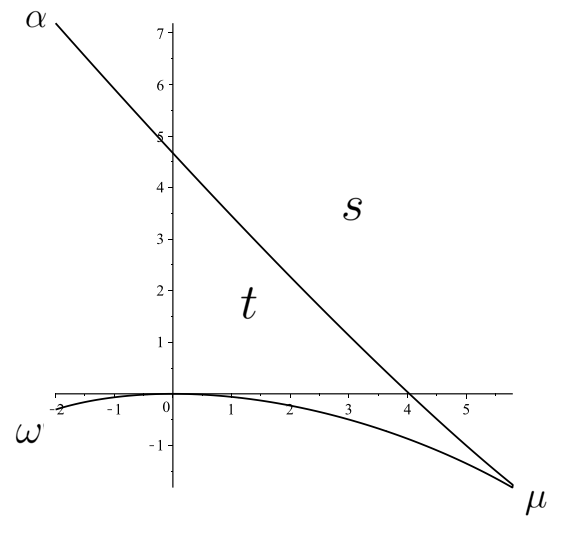}}}

    \caption{Zone G: The set $\Delta^{\sharp}$ for $a=-2$ and $b=-4$.}
\label{fig8}
\end{figure}

\begin{tabular}{lclllcllc}
&domain $s$&&&domain $t$ \\ \\

15&$\sigma_{3,1}$ \, (0,1)&\qquad \qquad \qquad& 20&$ \sigma_{3,1}$ \, (2,1)\\
16&$\sigma_{3,2}$ \, (0,1) &\qquad \qquad \qquad&21&$\sigma_{3,2}$ \, (2,1)  \\
17& $\sigma_{3,3}$ \, (1,0)&\qquad \qquad \qquad&22&$\sigma_{3,3}$ \, (1,2)  \\
18& $\sigma_{3,4}$ \, (1,0)&\qquad \qquad \qquad&28&$\sigma_{3,4}$ \, (3,0)   \\ \\

\end{tabular} \medskip

When comparing Fig.~\ref{fig7} and Fig. ~\ref{fig8} it becomes clear that for $a=-2$ and for some $b=b_*\in (-4,-2.5)$, in the corresponding picture of the set $\Delta ^{\sharp}$, the domain $h$ intersects only the third, but not the second quadrant.


\begin{figure}[H]
\vskip0.5cm
\centerline{\hbox{\includegraphics[scale=0.3]{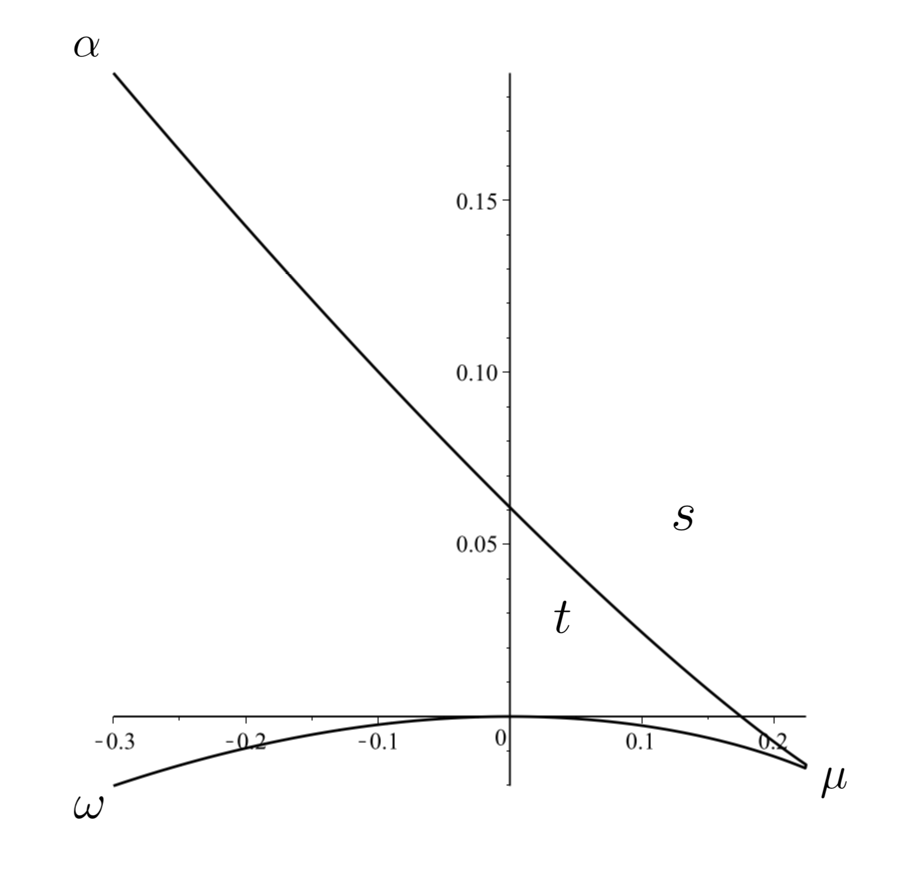}}}
    \caption{Zone H: The set $\Delta^{\sharp}$ for $a=1$ and $b=-1$. }
\label{fig11}
\end{figure}

\begin{tabular}{lclllcllc}
&domain $s$&&&domain $t$ \\ \\

30&$\sigma_{4,1}$ \, (0,1)&\qquad \qquad \qquad& 34&$\sigma_{4,1}$ \, (2,1)  \\
31&$\sigma_{4,2}$ \, (0,1) &\qquad \qquad \qquad&35&$\sigma_{4,2}$ \, (2,1)    \\
32&$\sigma_{4,3}$ \, (1,0)&\qquad \qquad \qquad&36&$\sigma_{4,3}$ \, (1,2)   \\
33&$\sigma_{4,4}$ \, (1,0)&\qquad \qquad \qquad&37&$\sigma_{4,4}$ \, (3,0)  \\ \\

\end{tabular} \medskip


\begin{figure}[H]
\vskip0.5cm
\centerline{\hbox{\includegraphics[scale=0.32]{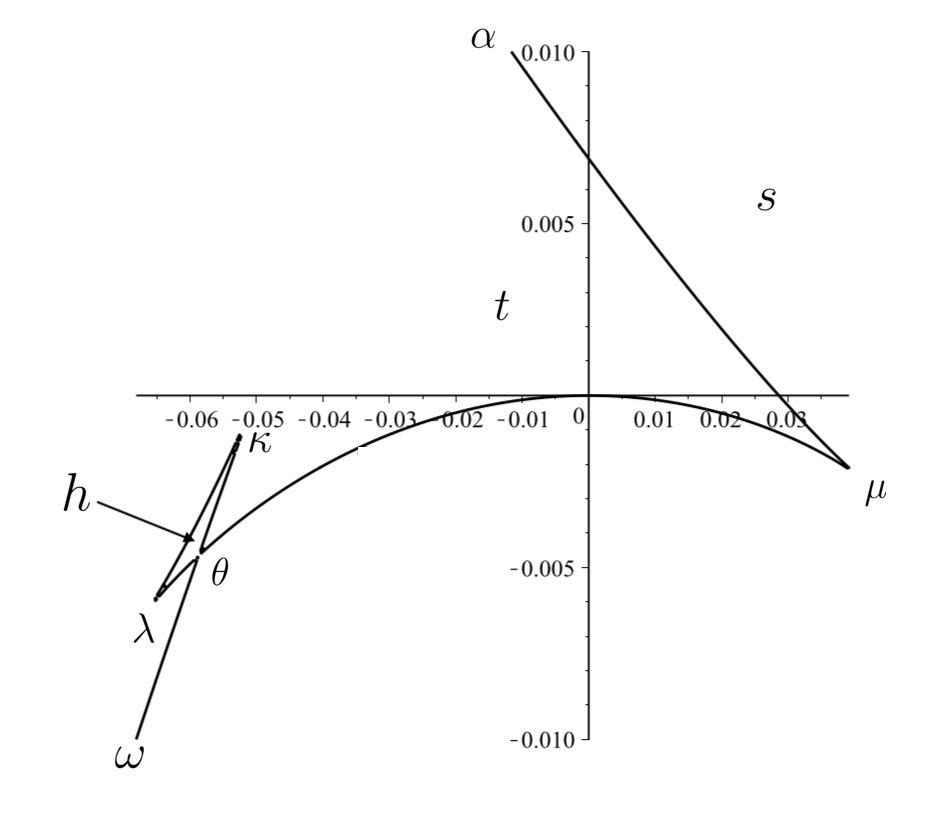}}}
    \caption{Zone I: The set $\Delta^{\sharp}$ for $a=0.05$ and $b=-0.2$. }
\label{fig12}
\end{figure}

\begin{tabular}{lclllcllc}
&domain $s$&&&domain $t$&&&domain $h$ \\ \\

30&$\sigma_{4,1}$ \, (0,1)&\qquad \qquad \qquad&34&$ \sigma_{4,1}$ \, (2,1)&\qquad \qquad \qquad&     \\
31&$\sigma_{4,2}$ \, (0,1)&\qquad \qquad \qquad&35&$\sigma_{4,2}$ \, (2,1)  \\
32&$\sigma_{4,3}$ \, (1,0) &\qquad \qquad \qquad&36&$\sigma_{4,3}$ \, (1,2)&&38&$\sigma_{4,3}$ \, (1,4)  \\
33&$\sigma_{4,4}$ \, (1,0) &\qquad \qquad \qquad&37&$\sigma_{4,4}$ \, (3,0) \\ \\
%

\end{tabular} \medskip


\begin{figure}[H]
\vskip0.5cm
\centerline{\hbox{\includegraphics[scale=0.3]{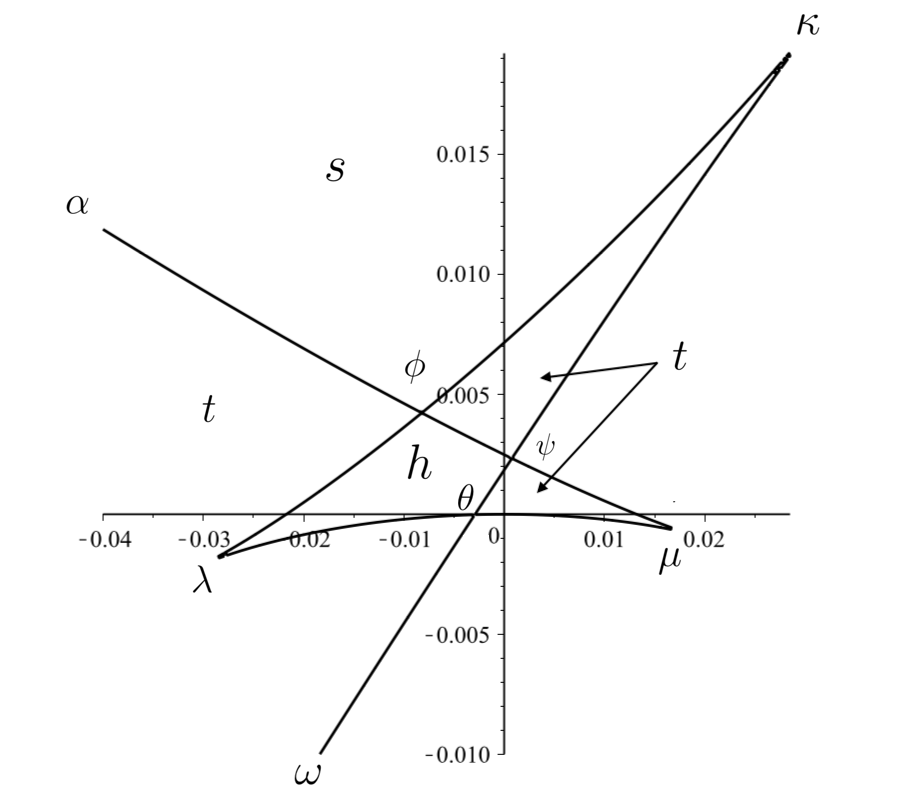}}}
    \caption{Zone J: The set $\Delta^{\sharp}$ for $a=0.05$ and $b=-0.12$. }
\label{fig13}
\end{figure}

\begin{tabular}{lclllcllc}
&domain $s$&&&domain $t$&&&domain $h$ \\ \\

30&$\sigma_{4,1}$ \, (0,1)&&34 , 39&$ \sigma_{4,1}$ \, (2,1) , (0,3) &\qquad \qquad \qquad &41&$\sigma_{4,1}$ \, (2,3) \\
31&$\sigma_{4,2}$ \, (0,1)&\qquad \qquad \qquad&35 , 40&$\sigma_{4,2}$ \, (2,1) , (0,3) &&42&$\sigma_{4,2}$ \, (2,3)  \\
32&$\sigma_{4,3}$ \, (1,0) &\qquad \qquad \qquad&36&$\sigma_{4,3}$ \, (1,2) &&38&$\sigma_{4,3}$ \, (1,4)  \\
33& $\sigma_{4,4}$ \, (1,0)&\qquad \qquad \qquad&37&$\sigma_{4,4}$ \, (3,0)  && \\ \\

\end{tabular} \medskip

By comparing Figures~\ref{fig12} and \ref{fig13} it becomes clear that for $a=0.05$ and 
for some
$b=b^{\flat}\in (-0.2, -0.12)$, the domain $h$ intersects
the second and the third, but not the first quadrant.


\begin{figure}[H]
\vskip0.5cm
\centerline{\hbox{\includegraphics[scale=0.32]{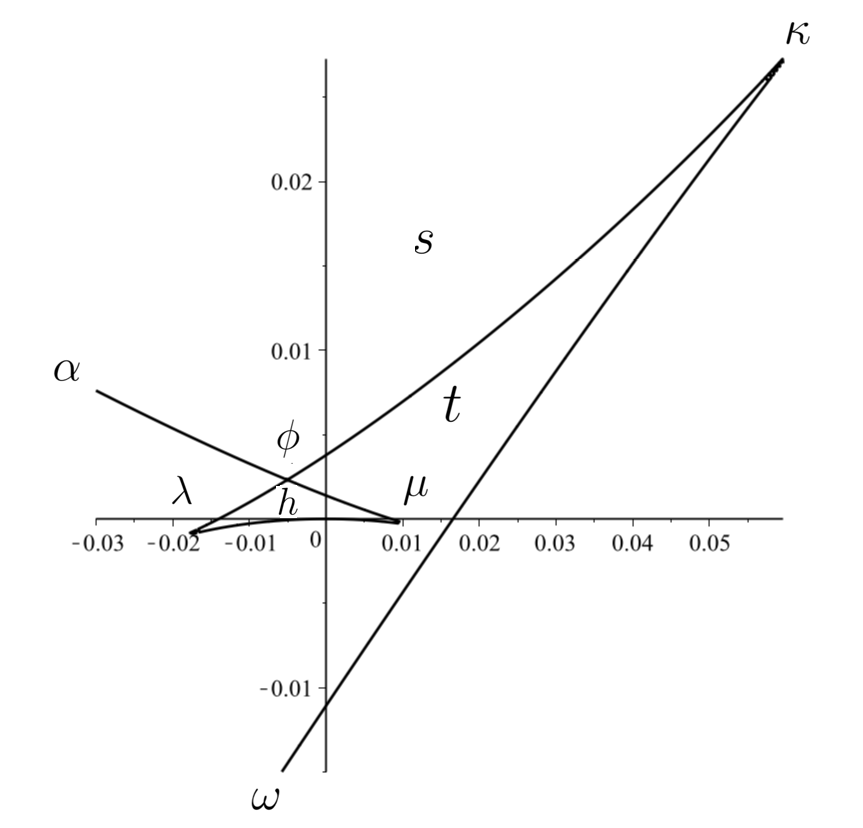}}}
    \caption{Zone K: The set $\Delta^{\sharp}$ for $a=0.05$ and $b=-0.09$. }
\label{fig14}
\end{figure}

\begin{tabular}{lclllcllc}
&domain $s$&&&domain $t$&&&domain $h$ \\ \\

30&$\sigma_{4,1}$ \, (0,1)&&39&$ \sigma_{4,1}$ \, (0,3) &\qquad \qquad \qquad&41&$\sigma_{4,1}$ \, (2,3)   \\
31&$\sigma_{4,2}$ \, (0,1)&\qquad \qquad \qquad&35 , 40&$\sigma_{4,2}$ \, (2,1) , (0,3)&&42&$\sigma_{4,2}$ \, (2,3)   \\
32&$\sigma_{4,3}$ \, (1,0) &\qquad \qquad \qquad&36&$\sigma_{4,3}$ \, (1,2)&&38&$\sigma_{4,3}$ \, (1,4)   \\
33&$\sigma_{4,4}$ \, (1,0) &\qquad \qquad \qquad&43&$\sigma_{4,4}$ \, (1,2) &&44&$\sigma_{4,4}$ \, (3,2)  \\ \\

\end{tabular} \medskip


\begin{figure}[H]
\vskip0.5cm
\centerline{\hbox{\includegraphics[scale=0.32]{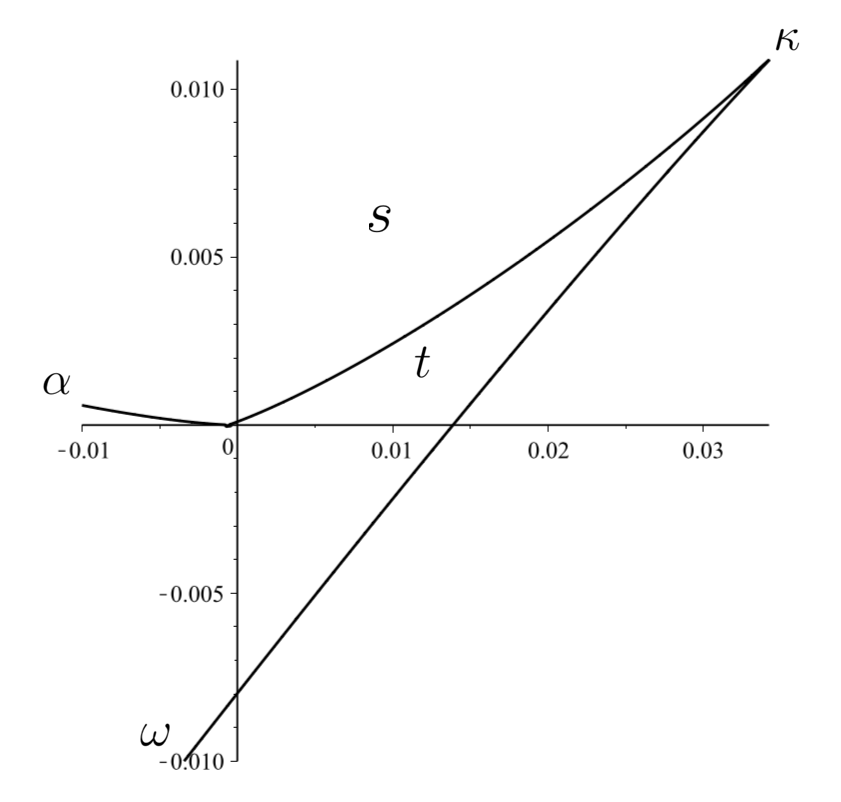}}\hskip1cm \hbox{\includegraphics[scale=0.32]{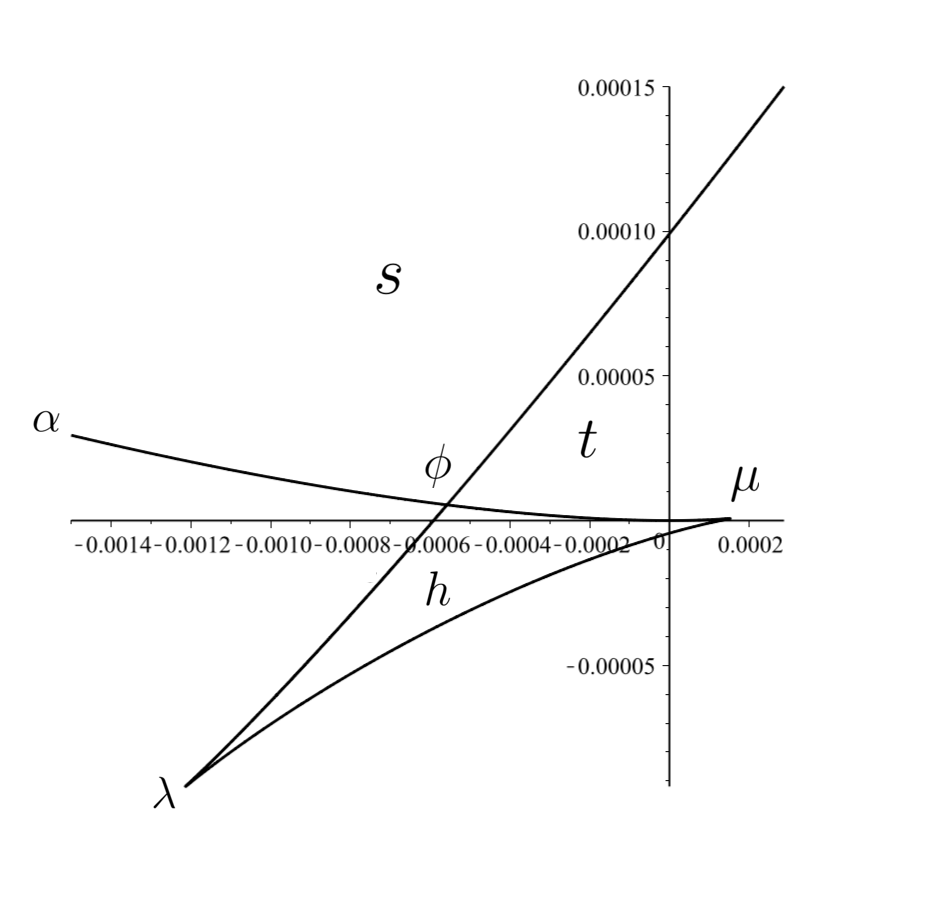}}}
    \caption{Zone L: The set $\Delta^{\sharp}$ for $a=0.22$ and $b=0.01$.} 
    
\label{fig15}
\end{figure}

\begin{tabular}{lclllcllc}
&domain $s$&&&domain $t$&&&domain $h$ \\ \\

45&$\sigma_{1,1}$ \, (0,1)&&49&$ \sigma_{1,1}$ \, (0,3) &\qquad \qquad \qquad&54&$\sigma_{1,1}$ \, (0,5)  \\
 46&$\sigma_{1,2}$ \, (0,1)&\qquad \qquad \qquad&50 , 51&$\sigma_{1,2}$ \, (0,3) , (2,1)&&55&$\sigma_{1,2}$ \, (2,3) \\
47& $\sigma_{1,3}$ \, (1,0) &\qquad \qquad \qquad&52&$\sigma_{1,3}$ \, (1,2) &&56&$\sigma_{1,3}$ \, (1,4)  \\
48&$\sigma_{1,4}$ \, (1,0)&\qquad \qquad \qquad&53&$\sigma_{1,4}$ \, (1,2)&&57&$\sigma_{1,4}$ \, (1,4)   \\

\end{tabular} \medskip


\begin{figure}[H]
\vskip0.5cm
\centerline{\hbox{\includegraphics[scale=0.32]{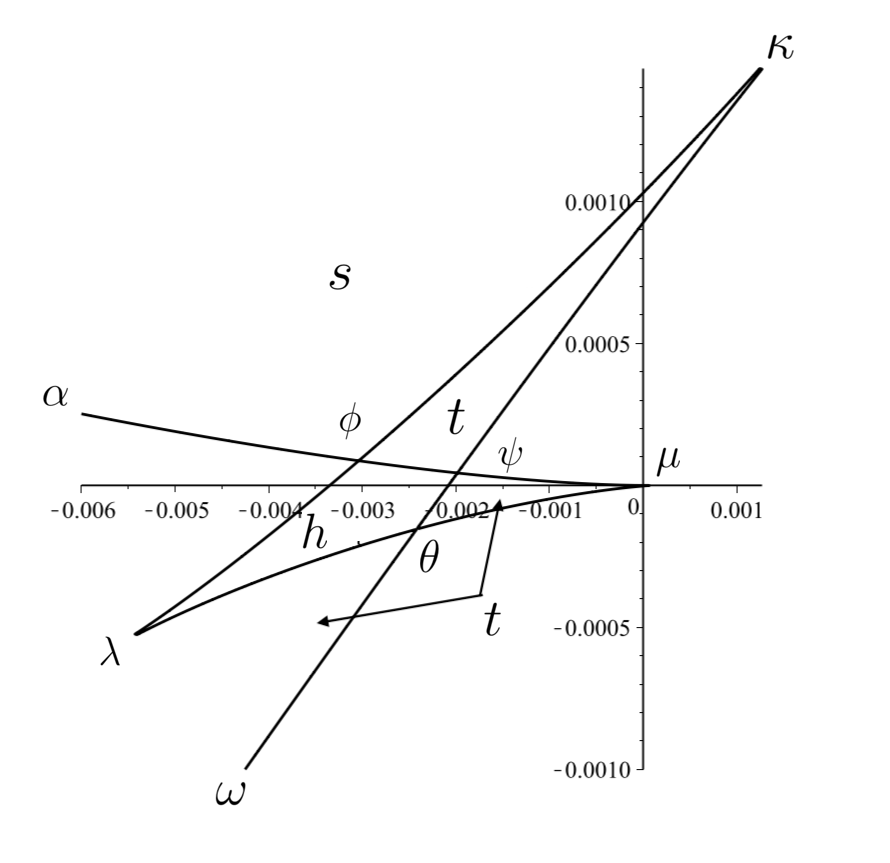}}\hskip1cm \hbox{\includegraphics[scale=0.32]{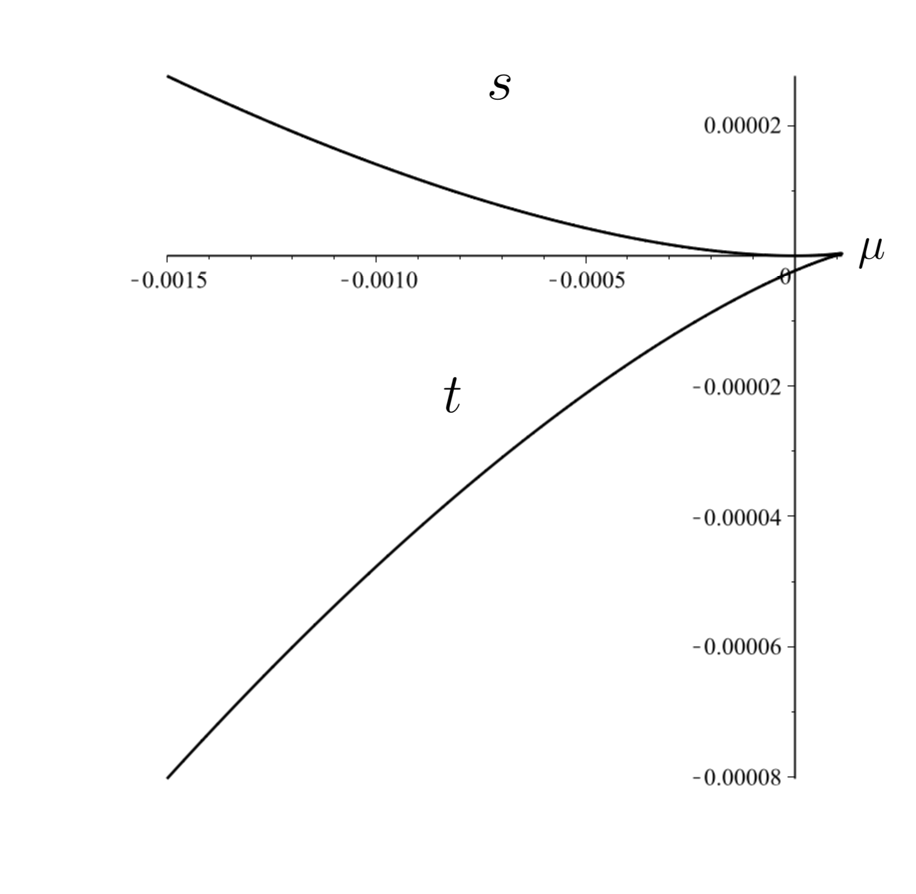}}}
    \caption{Zone M: The set $\Delta^{\sharp}$ for $a=0.28$ and $b=0.01$.} 
    
\label{fig16}
\end{figure}

\begin{tabular}{lclllcllc}
&domain $s$&&&domain $t$&&&domain $h$ \\ \\

45&$\sigma_{1,1}$ \, (0,1)&\qquad \qquad \qquad&49&$ \sigma_{1,1}$ \, (0,3) &\qquad \qquad \qquad&   \\
46&$\sigma_{1,2}$ \, (0,1)&\qquad \qquad \qquad&50 , 51&$\sigma_{1,2}$ \, (0,3) , (2,1)&&55&$\sigma_{1,2}$ \, (2,3)  \\
47&$\sigma_{1,3}$ \, (1,0)&\qquad \qquad \qquad&52&$\sigma_{1,3}$ \, (1,2) &&56&$\sigma_{1,3}$ \, (1,4)  \\
48& $\sigma_{1,4}$ \, (1,0)&\qquad \qquad \qquad&53& $\sigma_{1,4}$ \, (1,2)&&  \\ \\

\end{tabular} \medskip


\begin{figure}[H]
\vskip0.5cm
\centerline{\hbox{\includegraphics[scale=0.3]{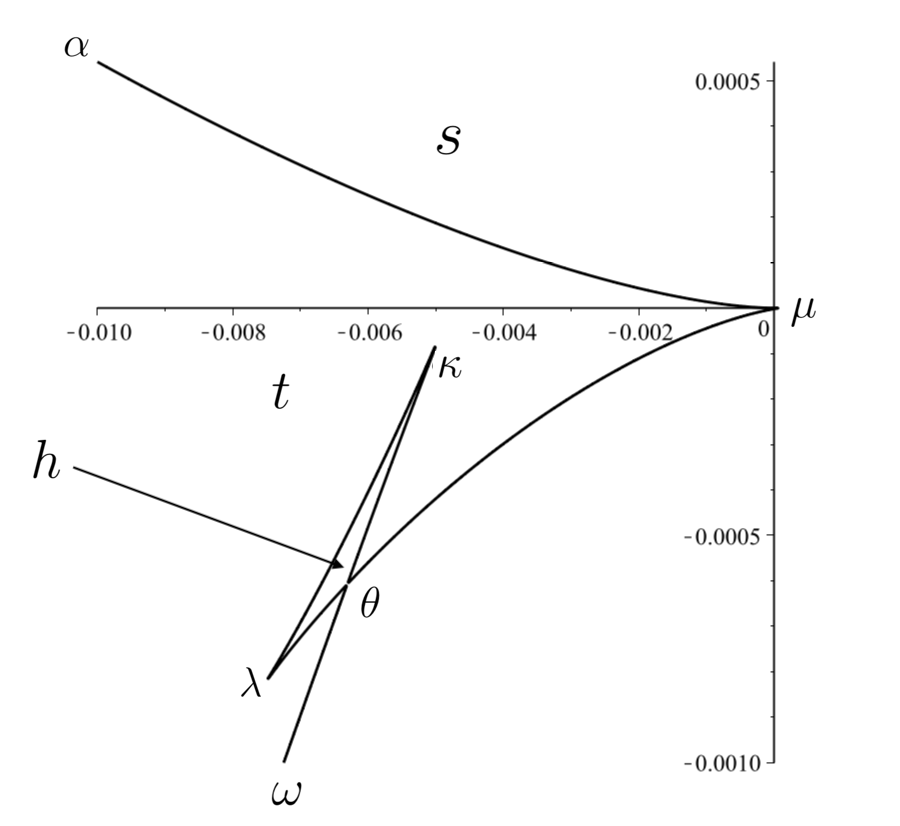}}\hskip1cm \hbox{\includegraphics[scale=0.3]{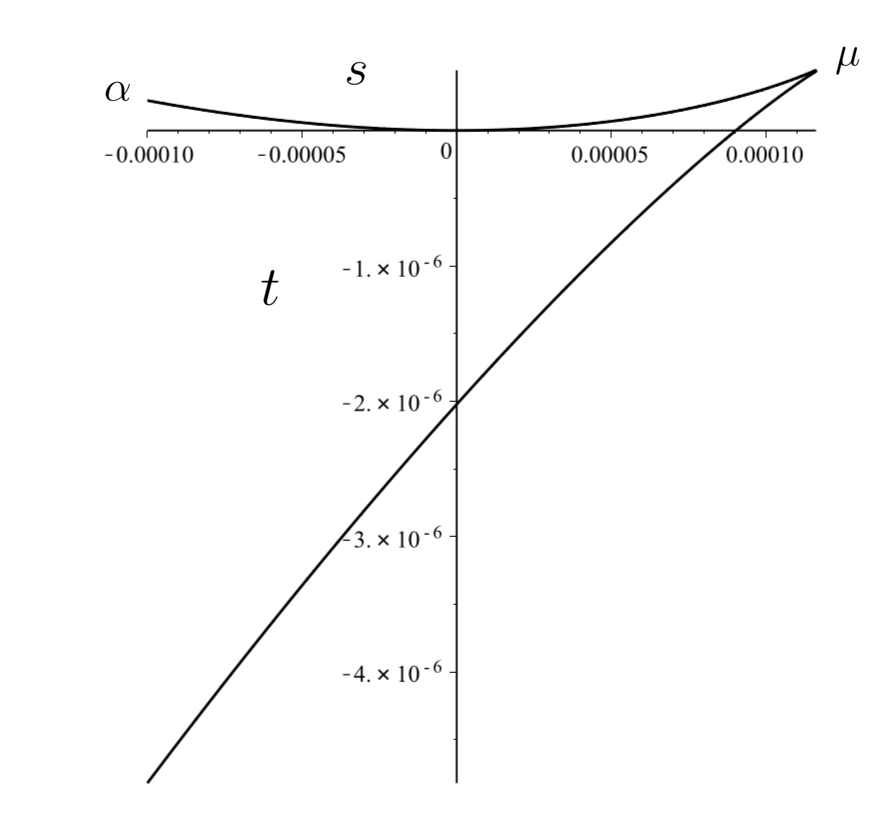}}}
    \caption{Zone N: The set $\Delta^{\sharp}$ for $a=0.295$ and $b=0.01$.}
    
\label{fig17}
\end{figure}

\begin{tabular}{lclllcllc}
&domain $s$&&&domain $t$&&&domain $h$ \\ \\

45&$\sigma_{1,1}$ \, (0,1)&\qquad \qquad \qquad&49&$\sigma_{1,1}$ \, (0,3)&\qquad \qquad \qquad \\
46& $\sigma_{1,2}$ \, (0,1)&\qquad \qquad \qquad&51&$\sigma_{1,2}$ \, (2,1)  \\
47&$\sigma_{1,3}$ \, (1,0)&\qquad \qquad \qquad&52&$\sigma_{1,3}$ \, (1,2) &&56&$\sigma_{1,3}$ \, (1,4)  \\
48&$\sigma_{1,4}$ \, (1,0)&\qquad \qquad \qquad&53&$\sigma_{1,4}$ \, (1,2)   \\ \\
%

\end{tabular} \medskip

When comparing Figures \ref{fig16} and \ref{fig17} it becomes clear that for $b=0.01$, 
there exist two values
$0.28<a_{\dagger}<a_{\natural}<0.295$ of $a$ such that for 
$a=a_{\dagger}$, the cusp $\kappa$ is in the
second quadrant, but still above the arc $\mu \alpha$ (so the 
self-intersection points $\phi$ and $\psi$
exist), and for $a=a_{\natural}$, the cusp point $\kappa$ is in the 
second quadrant and
under the arc $\mu \alpha$.


\begin{figure}[H]

\vskip0.5cm
\centerline{\hbox{\includegraphics[scale=0.3]{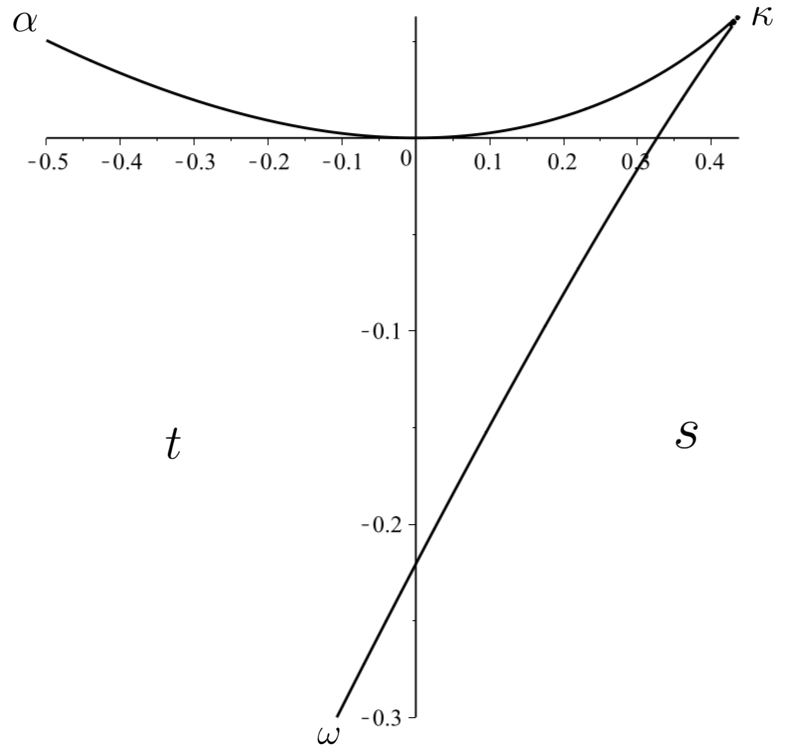}}}
    \caption{Zone P: The set $\Delta^{\sharp}$ for $a=1$ and $b=1$. }
\label{fig10}
\end{figure}

\begin{tabular}{lclllcllc}
&domain $s$&&&domain $t$&&& \\ \\

45&$\sigma_{1,1}$ \, (0,1)&\qquad \qquad \qquad&49&$ \sigma_{1,1}$ \, (0,3) \\
46&$\sigma_{1,2}$ \, (0,1)&\qquad \qquad \qquad&51&$\sigma_{1,2}$ \, (2,1)   \\
47&$\sigma_{1,3}$ \, (1,0) &\qquad \qquad \qquad& 52&$\sigma_{1,3}$ \, (1,2)\\
48&$\sigma_{1,4}$ \, (1,0)&\qquad \qquad \qquad&53&$\sigma_{1,4}$ \, (1,2)  \\  \\

\end{tabular} \medskip



\begin{thebibliography}{Dillo 83}

\bibitem{AlFu} A.~Albouy and Y.~Fu, Some remarks about Descartes' rule 
of signs. Elem. Math., 69 (2014), 186--194. Zbl 1342.12002, 
MR3272179.

\bibitem{Arnold} V. I. Arnold, S. M. Gusein-Zade and A. N. Varchenko, Singularities of Differentiable Maps, Birkh\"auser; 2012 edition (May 24, 2012). 



\bibitem{YVH2s} H.~Cheriha, Y.~Gati and V.~P.~Kostov, A nonrealization theorem in the context of Descartes' rule of signs, St. Kliment Ohridski - Faculty of Mathematics and Informatics (accepted).




\bibitem{YVHd9} H.~Cheriha, Y.~Gati and V.~P.~Kostov, On Descartes' rule for polynomials with two variations of signs (submitted).

\bibitem{Che} B.~Chevallier, Courbes maximales de Harnack et discriminant. (French) [Maximal Harnack curves and the discriminant] S\'eminaire sur la g\'eom\'etrie alg\'ebrique r\'eelle, Tome I, II, 41-65, Publ. Math. Univ. Paris VII, 24, Univ. Paris VII, Paris, 1986.







\bibitem{FoKoSh} J.~Forsg\aa rd, B.~Shapiro and V.~P.~Kostov, 
Could Ren\'e Descartes have known this?  Exp. Math. 24 (4)  
(2015), 
438-448. Zbl 1326.26027, MR3383475.

\bibitem{corVS} J.~Forsg\aa rd, B.~Shapiro and V.~P.~Kostov, 
Corrigendum: Could Ren\'e Descartes have known this? Exp. Math. 28 (2) (2019), 255-256. 10.1080/10586458.2017.1417775.

\bibitem{Fo}  J.~Fourier, Sur l'€™usage du th\'eor\`eme de Descartes dans 
la recherche
des limites des racines. Bulletin des sciences par la Soci\'et\'e philomatique
de Paris (1820) 156--165, 181--187; {\oe}uvres 2,  291--309, Gauthier - Villars, 1890.

\bibitem{Gr} D.~J.~Grabiner, Descartes' Rule of Signs: 
Another Construction. Am. Math. Mon. 106 (1999), 854--856.
Zbl 0980.12001, MR1732666.


\bibitem{VJ} V.~Jullien, Descartes La "Geometrie" de 1637.


\bibitem{VMJ} V.~P.~Kostov, On a stratification defined by real roots of polynomials, 
Serdica
Mathematical Journal 29:2 (2003) 177-186.

\bibitem{Czechoslovak} V.~P.~Kostov, On realizability of sign patterns 
by real polynomials, Czechoslovak Math. J. 68 (2018) 853-874.
%






\bibitem{KostovShapiroActaUMB} V.~P.~Kostov and B.~Shapiro, Polynomials, sign patterns and Descartes' 
rule, Acta Universitatis Matthiae Belii, series Mathematics
Issue 2019, 1-11, ISSN 1338-7111.


\bibitem{Me} I. M\'eguerditchian, Thesis - G\'eom\'etrie du discriminant r\'eel et des polyn\^omes hyperboliques, thesis defended in 1991 at the University Rennes 1.









\bibitem{Pos} T.~Poston and I.~Stewart, Catastrophe theory and its applications. With an appendix by D. R. Olsen, S. R. Carter and A. Rockwood. Surveys and Reference Works in Mathematics, No. 2. Pitman, London-San Francisco, Calif.-Melbourne: distributed by Fearon-Pitman Publishers, Inc., Belmont, Calif., 1978. xviii+491 pp. ISBN: 0-273-01029-8.






















%


\end{thebibliography}
\end{document}